\documentclass[12pt]{article}
\usepackage{graphicx,amsmath,amsfonts,amssymb,amscd,amsthm, mathrsfs, bm}
\newtheoremstyle{kai}
{3pt}{3pt}{}{}{\bfseries}{.}{.5em}{}
\makeatletter
\def\EquationsBySection{\def\theequation
{\thesection.\arabic{equation}}%
\@addtoreset{equation}{section}}

\newcommand\old[1]{}
\newcommand{\pend}{\hfill \thicklines \framebox(6.6,6.6)[l]{}}
\renewenvironment{proof}{\noindent {\it  Proof.} \rm}{\pend}
\newtheorem{theorem}{Theorem}[section]
\newtheorem{lemma}{Lemma}[section]
\newtheorem{corollary}{Corollary}[section]
\newtheorem{proposition}{Proposition}[section]
\newtheorem{remark}{Remark}[section]
\newtheorem{definition}{Definition}[section]
\newtheorem{example}{Example}[section]

\EquationsBySection
\makeatother

\begin{document}
\pagestyle{plain}
\title
{\bf Stationarity of Stochastic Linear Equations  with Additive Noise and  Delays in the Unbounded Drift Terms}
\author{
\\
Kai Liu
\\
\\
\small{Department of Mathematical Sciences,}\\
\small{The University of Liverpool,}\\
\small{Peach Street, Liverpool, L69 7ZL, U.K.}\\
\small{E-mail: k.liu@liv.ac.uk}\\}

\date{}
\maketitle

\noindent {\bf Abstract:} This paper continues the study of \cite{kl08(2), kl11} for
 stationary solutions of stochastic linear retarded functional differential equations with the emphasis on delays  which appear in those terms including  spatial partial derivatives. As a consequence, the associated stochastic equations have unbounded operators acting on the discrete or distributed delayed terms, while the operator acting on the instantaneous term generates a strongly continuous semigroup. We present conditions on the delay systems to obtain a unique stationary solution  by combining spectrum analysis of unbounded operators and stochastic calculus. A few instructive cases are analyzed in detail to clarify the underlying complexity in the study of systems with unbounded delayed operators. 

\vskip 10pt

\noindent{\bf 2000 Mathematics Subject Classification(s):} 60H15, 60G15, 60H05.

\newpage
\section{Introduction}

Let $X$ be a separable  real Banach space with  norm $\|\cdot\|_X$ and $A:\, {\mathscr D}(A)\subset X\to X$ generates a strongly continuous semigroup $e^{tA}$, $t\ge 0$, on $X$. Suppose that $Z$ is another Banach space, equipped with the norm $\|\cdot\|_Z$, such that ${\mathscr D}(A)\hookrightarrow Z\hookrightarrow  X$, i.e., the injection $\hookrightarrow$ is dense and continuous.  
Let  $W= (Z, X)_{1/2, 2}$ be the standard real interpolation space  between $Z$ and $X$ (see, e.g., \cite{ht1997}). If $Z=X$, then we take  $W=X$.  
Let $r\ge 0$ and  $L^2_r := L^2([-r, 0]; Z)$. We denote by ${\cal X}$  the product  space $W\times L^2_r$ with  norm
\[
\|\Phi\|_{{\cal X}}=\|\phi_0\|_W + \|\phi_1\|_{L^2_r}\hskip 15pt \hbox{for all}\hskip 15pt\Phi= (\phi_0,\phi_1)\in {\cal X}.\] 

Consider  the following system which is described by a stochastic linear retarded functional differential equation on $X$,
\begin{equation}
\label{17/05/06(1)}
\begin{cases}
dy(t) = Ay(t)dt + Fy_t dt+ f(t)dB(t),\hskip 15pt t\ge 0,\\
y(0) = \phi_0,\hskip 5pt y_0=\phi_1,\hskip 5pt \Phi=  (\phi_0, \phi_1)\in{\cal X},\\
\end{cases}
\end{equation} 
where $y_t(\theta) := y(t+\theta)$, called {\it (history) segment},  for any $\theta\in [-r, 0]$ and $t\ge 0$, $f$ is an appropriate function and $B$ is a Brownian motion defined on some probability space $(\Omega, {\mathscr F}, {\mathbb P})$. Here the delay term $F:\,C([-r, 0]; Z)\to X$ is a  bounded linear operator which admits the following representation
\begin{equation}
\label{14/10/13(1)}
F\varphi = \int^0_{-r} d\eta(\theta)\varphi(\theta)\hskip 20pt \forall\,\varphi \in C([-r, 0]; Z),
\end{equation}
where $\eta:\,[-r, 0]\to {\mathscr L}(Z, X)$, the family of all bounded and linear operators from $Z$ to $X$,  is of bounded variation. 

Although operator $F$ is defined only on continuous functions, the quantity $Fy_t$ still makes sense as function of $t$ with values in $X$ for each $y(\cdot)$ in $L^2([-r, T]; Z)$. Indeed, we have the following result whose proof is referred to Appendix.
\begin{proposition}
\label{14/10/13(10)}
Let $T\ge 0$ and $y(\cdot)\in L^2([-r, T]; Z)$, then the function $t\to Fy_t$ belongs to $L^2([0, T]; X)$. Moreover, there exists a constant $C>0$ such that 
\begin{equation}
\label{14/10/13(3)}
\int^T_0 \|Fy_t\|^2_X dt \le C\int^T_{-r} \|y(t)\|^2_Z dt.
\end{equation}
\end{proposition}

A typical example satisfying (\ref{14/10/13(1)}) and thus Proposition \ref{14/10/13(10)} is given below.  
Assume that $\eta$ is the Stieltjes measure defined by 
\begin{equation}
\label{19/11/07(12)}
\eta(\tau)=-\sum^m_{i=1}{\bf 1}_{(-\infty, \,-r_i]}(\tau)A_i -\int^0_\tau A_0(\theta)d\theta,\hskip 15pt \tau\in [-r, 0],
\end{equation}
where  ${\bf 1}_{(-\infty, -r_i]}$ denotes the indicator function on $(-\infty, -r_i]$,
$0\le r_i \le r$, $A_i\in {\mathscr L}({\mathscr D}(A), X)$, $i=1,\cdots, m$, and $A_0(\cdot)\in L^2([-r, 0]; {\mathscr L}({\mathscr D}(A), X))$. Let $Z={\mathscr D}(A)$, endowed with the grath norm of $A$, and define a linear mapping ${F}: C([-r, 0]; {\mathscr D}(A))\to X$ by
\begin{equation}
\label{19/11/07(13)}
{F}\varphi =\int^0_{-r} d\eta(\theta)\varphi(\theta)=\sum^m_{i=1}A_i\varphi (-r_i) + \int^0_{-r} A_0(\theta)\varphi(\theta)d\theta,\hskip 20pt \forall\, \varphi\in C([-r, 0]; {\mathscr D}(A)).
\end{equation}
It is clear that $F:\, C([-r, 0]; {\mathscr D}(A)) \to X$ is linear and bounded.  For any fixed $T\ge 0$, $y\in C([-r, T]; {\mathscr D}(A))$, one can easily derive by using H\"older inequality and Fubini's theorem that
\[
\begin{split}
\Big(\int^{T}_0 &\|Fy_s\|^2_Xds\Big)^{1/2}\\
&\le \Big[\sum^m_{i=1} \|A_i\|_{{\mathscr L}({\mathscr D}(A), X)} + \Big(\int^0_{-r} \|A_0(\theta)\|^2_{{\mathscr L}({\mathscr D}(A), X)}d\theta\Big)^{1/2}\cdot r^{1/2}\Big] \Big(\int^{T}_{-r} \|y(s)\|_{{\mathscr D}(A)}^2 ds\Big)^{1/2}.
\end{split}
\]
Since $C([-r, T]; {\mathscr D}(A))$ is dense in $L^2([-r, T]; {\mathscr D}(A))$,  the delay operator ${F}$  is extendible so that (\ref{14/10/13(3)}) (here, $Z={\mathscr D}(A)$) is valid for all $y\in L^2([-r, T]; {\mathscr D}(A))$.

If $Z=X$,  the associated delay operator $F$ is  bounded, a case considered in  \cite{kl08(2), kl09(2), kl11}. If $Z\not= X$, we deal with, in essence, unbounded delay terms. In this case, we futher assume that $A$ generates an analytic semigroup $e^{tA}$, $t\ge 0$, on appropriate spaces and  meanwhile employ  the theory of interpolation spaces.

\begin{example}\rm  
Let $X=H$ be a Hilbert space and  $A$ generate an analytic semigroup $e^{tA}$, $t\ge 0$, on $H$. Consider a Stieltjes measure $\eta$ given by 
\begin{equation}
\label{11/08/2013(10655)}
\eta(\theta) = -{\bf 1}_{(-\infty, -r]}(\theta)\alpha A_1 - \int^0_\theta \beta(\tau)A_2d\tau:\, {\mathscr D}(A)\to H,\hskip 15pt \theta\in [-r, 0],  
\end{equation}
where $\alpha\in {\mathbb R}$ and the real-valued function $\beta(\cdot)$ is assumed to be $L^2$-integrable on $[-r, 0]$, i.e., $\beta\in L^2([-r, 0]; {\mathbb R})$. The delayed operator $F$  is explicitly written as 
\[
Fy_t = \int^0_{-r}d\eta(\theta)y(t+\theta) =   \alpha A_1y(t-r) + \displaystyle\int^0_{-r}\beta(\theta)A_2y(t+\theta)d\theta,\hskip 20pt t\ge 0.\]
In this case, we put $Z={\mathscr D}(A)$ and let $W$ denote the intermediate space $({\mathscr D}(A), H)_{1/2, 2}$ between ${\mathscr D}(A)$ and $H$ given by:
\[
W=\Big\{x\in H:\, \int^\infty_0 \|Ae^{tA}x\|^2_H dt<\infty\Big\}\]
and 
\[
\|x\|_W =\Big(\|x\|^2_H + \int^\infty_0 \|Ae^{tA}x\|^2_H dt\Big)^{1/2},\hskip 20pt x\in W.\]
  In particular, we have ${\cal X}=W\times L^2([-r, 0]; {\mathscr D}(A))$. 

For example,  consider  an initial-boundary value problem of Dirichlet type for the stochastic retarded Laplace equation:
\begin{equation}
\label{11/10/2013(2)}
\begin{cases}
\displaystyle\frac{\partial y(t, x)}{\partial t} = \Delta y(t, x) + \gamma \Delta y(t-r, x) + \mu \displaystyle\int^0_{-r} \Delta y(t+\theta, x)d\theta + f(t, x) \dot B(t)\,\,\,\hbox{on}\,\,\, [0, T]\times {\cal O},\\
y(t, x) =y_0(t, x),\hskip 20pt (t, x)\in [-r, 0]\times {\cal O},\\
y(0, x)=\varphi(x),\hskip 20pt x\in {\cal O},\\
y(t, x)=0,\hskip 20pt (t, x)\in [-r, T]\times \partial{\cal O},.
\end{cases}
\end{equation}
Here ${\cal O}$ is a bounded open subset of ${\mathbb R}^n$ with smooth boundary $\partial {\cal O}$, $\gamma$, $\mu\in {\mathbb R}$, $r>0$, $T>0$ and $y_0$ and $\varphi$ are appropriately given functions.
We can rewrite  (\ref{11/10/2013(2)}) as an initial boundary problem (\ref{17/05/06(1)}) in the Hilbert space $X=L^2({\cal O})$ by setting
\begin{equation}
\begin{cases}
A= \Delta,\\
{\mathscr D}(A) = W^{2, 2}({\cal O})\cap W^{1, 2}_0({\cal O}),\\
A_1 =\gamma \Delta,\hskip 20pt A_2 =\mu\Delta.
\end{cases}
\end{equation}
On this occasion, the interpolation space $({\mathscr D}(A), X)_{1/2, 2}$ is equivalent to $W^{1, 2}_0({\cal O})$. 
\end{example}

\begin{example}\rm
 Assume that $V$, $H$ are two Hilbert spaces such that 
\[
V\hookrightarrow H\cong H^*\hookrightarrow V^*.\]
Let $a(u, v)$ be a bounded sesquilinear form defined on $V\times V$ satisfying G\aa rding's inequality 
\begin{equation}
\label{29/12/11(10)}
2a(u, u)\le -\delta\|u\|^2_V,\hskip 20pt u\in V,
\end{equation}
where $\delta>0$ is a  constant. Let $A$ be the operator associated with this sesquilinear form by
\begin{equation}
\label{29/12/11(11)}
 \langle v, Au\rangle_{V, V^*}=a(u, v),\,\,\,\,\,u,\,\,v\in V.
\end{equation}
Then operator $A$ is bounded and linear from $V$ into $V^*$. The realization of $A$ in $H$, which is the restriction of $A$ to the domain ${\mathscr D}(A)=\{v\in V:\, Av\in H\}$, is also denoted by $A$. It is known (cf. \cite{ht1979})  that $A$ generates a bounded analytic semigroup $e^{tA}$, $t\ge 0$, on $V^*$ and $e^{tA}:\, V^*\to V$ for each $t>0$. 

Let $X=V^*$, $Z=V$ and  $W=(V, V^*)_{1/2, 2}=H$. In this case, we have ${\cal X}=H\times L^2([-r, 0]; V)$.
Let $A_i\in {\mathscr L}(V, V^*)$, $i=1,\,2$, such that $A_i$ maps ${\mathscr D}(A)$ endowed with the grath norm of $A$ into $H$ continuously.  
Consider a Stieltjes measure $\eta$ given by 
\begin{equation}
\label{11/08/2013(1068)}
\eta(\theta) = -{\bf 1}_{(-\infty, -r]}(\theta)\alpha A_1 - \int^0_\theta \beta(\tau)A_2d\tau:\, V\to V^*,\hskip 15pt \theta\in [-r, 0],  
\end{equation}
where  $\alpha\in {\mathbb R}$ and the real-valued function $\beta(\cdot)$ is assumed to be $L^2$-integrable on $[-r, 0]$, i.e., $\beta\in L^2([-r, 0]; {\mathbb R})$. 

For example,  consider the following  initial-boundary value problem for a stochastic parabolic   differential equation with delay. Let ${\cal O}\subset {\mathbb R}^n$ be a bounded domain with smooth boundary $\partial {\cal O}$. We set $H=L^2({\cal O}; {\mathbb R})$ and $V=H^1_0({\cal O}; {\mathbb R})$. Let $a(u, v)$ be the sesquilinear form in $H^1_0({\cal O}; {\mathbb R})\times H^1_0({\cal O}; {\mathbb R})$ defined by
\begin{equation}
\label{08/08/2013(10)}
a(u, v) = \int_{\cal O}\Big\{\sum^n_{i,\,j=1}a_{ij}(x)\frac{\partial u}{\partial x_i}\frac{{\partial v}}{\partial x_j} + \sum^n_{i=1}b_i(x)\frac{\partial u}{\partial x_i} v + c(x)uv\Big\}dx,\,\,\,\,x\in {\cal O}.
\end{equation}
Here  we assume that the real-valued coefficients $a_{ij}$, $b_i$, $c$ satisfy
\[
a_{ij}=a_{ji}\in C^1(\bar{\cal O}; {\mathbb R}),\hskip 20pt b_i\in C^1(\bar{\cal O}; {\mathbb R}),\hskip 20pt c\in L^\infty({\cal O}; {\mathbb R}),\hskip 20pt 1\le i,\,j\le n,\]
and the uniform ellipticity
\begin{equation}
\label{02/09/13(1)}
\sum^n_{i,\,j=1}a_{ij}(x)y_iy_j\ge \delta\|y\|^2_{{\mathbb R}^n},\hskip 20pt \forall\, y=(y_1,\cdots, y_n)\in {\mathbb R}^n,\hskip 15pt x\in {\cal O},
\end{equation} 
for some constant $\delta>0$. As is well known (see e.g., Tanabe \cite{ht1979}), this sesquilinear form is bounded and the operator $A: H^1_0({\cal O}; {\mathbb R})\to H^{-1}({\cal O}; {\mathbb R})$ defined through (\ref{08/08/2013(10)}) has the following realization in $L^2({\cal O}; {\mathbb R})$. Let 
\[
\tilde A = -\sum^n_{i,\,j=1}\frac{\partial}{\partial x_j}\Big(a_{ij}(x)\frac{\partial}{\partial x_j}\Big) + \sum^n_{i=1}b_i(x)\frac{\partial}{\partial x_i} + c(x),\hskip 20pt x\in {\cal O},\]
be the associated uniformly elliptic differential operator of the second order. Next, let $A_i$, $i=1,\,2$, be the restriction  to $H^1_0({\cal O}; {\mathbb R})$ of the second order differential operator $-\tilde{A}_i$, $i=1,\,2$, given by
\[
\tilde{A}_i =-\sum^n_{i,\,j=1}\frac{\partial}{\partial x_j}\Big(\tilde a_{ij}(x)\frac{\partial}{\partial x_j}\Big) + \sum^n_{i=1}\tilde b_i(x)\frac{\partial}{\partial x_i} + \tilde c(x),\hskip 20pt  x\in {\cal O},\]
where 
\[
\tilde a_{ij} = \tilde a_{ji}\in C^1(\bar{\cal O}; {\mathbb R}),\hskip 20pt \tilde b_i\in C^1(\bar{\cal O}; {\mathbb R}),\hskip 20pt \tilde c\in L^\infty({\cal O}; {\mathbb R}),\hskip 20pt  1\le i,\,j\le n.\]
Thus each $A_i: H^1_0({\cal O}; {\mathbb R})\to H^{-1}({\cal O}; {\mathbb R})$ is bounded without the ellipticity condition (\ref{02/09/13(1)}). The following system of a stochastic parabolic  partial functional  differential equation and initial-boundary condition is covered
\begin{equation}
\label{10/02/2012(1)}
\begin{cases}
\displaystyle\frac{\partial y(t, x)}{\partial t} = \tilde A y(t, x) + \tilde A_1y(t-r, x) +\int^0_{-r}\beta(\theta)\tilde A_2y(t+\theta, x)d\theta + f(t, x)\dot B(t),\,t\ge 0,\,x\in {\cal O},\\
y(0, \cdot) =\phi_0(\cdot)\in L^2({\cal O}; {\mathbb R}),\,\,\,\,y(t, \cdot)=\phi_1(t, \cdot)\in H^1_0({\cal O}; {\mathbb R}),\,\,\,\,\hbox{a.e.}\,\,\,\,t\in [-r, 0),
\end{cases} 
\end{equation}
where  the kernel $\beta(\cdot)$ is assumed to be an element of $L^2([-r, 0]; {\mathbb R})$. 
\end{example}

In \cite{kl08(2), kl11}, we studied stationary solutions  for the following abstract stochastic retarded evolution equation on a Hilbert space $H,$
\begin{equation}
\label{08/08/2013(1)}
\begin{cases}
\displaystyle dy(t) =Ay(t)dt + A_1y(t-r)dt + \displaystyle\int^0_{-r} \beta(\theta)A_2 y(t+\theta)d\theta dt + f(t)dB(t),\,\,\,\,t\ge 0,\\
y(0)=\phi_0,\,\,\,\,y(\theta)=\phi_1(\theta),\,\,\,\theta\in [-r, 0],\,\,\,r>0,
\end{cases}
\end{equation} 
where both the operators $A_1$ and $A_2$ appearing on the delay terms are linear and {\it bounded\/} on $H.$ In this work, we continue the study of stationary solutions for the equation (\ref{08/08/2013(1)}) by taking {\it unbounded\/} $A_1$ and $A_2$ into consideration.
 
The objective of this  work is to study stability and stationary  (strong) solutions  for a class of retarded Ornstein-Uhlenbeck processes. Here we are especially concerned about the situation that delay  appears in those terms including partial derivatives of the analogous system equations to (\ref{10/02/2012(1)}). In
Section 2, we first present a theory of fundamental solutions involved with unbounded delay operators. This is a natural  generalization of those in the  theory of bounded operators  developed in  \cite{kl08(2), kl11}. Afterwards, this powerful tool is used  in Section 3 to derive a variation  of constants formula for the stochastic systems under investigation. By using the explicit form of the retarded Ornstein-Uhlenbeck processes, we shall develop in Section 4 a theory of stability and stationary solutions. To locate a stationary solution for our system, it is important to know when the associated ``lift-up" solution semigroup  is exponentially stable, a case which is quite complicated in contrast to its bounded delay counterpart. To clarify and illustrate  our  theory, we split our statement of Section 5 into two parts, Subsections 5.1 and 5.2, to consider the discrete and distributed delays separately. In contrast with bounded delay situation in \cite{kl08(2), kl11}, it turns out that we need  different methods to deal with these two kinds of delays.
  Finally, we add an Appendix to present the proofs of some results from deterministic functional differential equations.

\section{Fundamental Solutions}

Assume that $A\in {\mathscr L}(Z, X)$ and $A$  generates an analytic semigroup $e^{tA}$, $t\ge 0$, on $X$. When $Z=X$, we only suppose that $A$ generates a $C_0$-semigroup $e^{tA}$, $t\ge 0$,  on $X$.
Let $f\in L^2([0, T]; X)$ and consider the following retarded functional differential  equation,
\begin{equation}
\label{14/08/2013(50)}
\begin{cases}
\displaystyle\frac{dy(t)}{dt} = Ay(t) + Fy_t + f(t),\,\,\,\,t\in (0, T],\\
 y(0)=\phi_0,\,\, y(t)=\phi_1(t),\,\,t\in [-r, 0],\,\,\,\Phi=(\phi_0, \phi_1)\in {\cal X}=W\times L^2([-r, 0]; Z),
\end{cases}
\end{equation}
or its integral form,
\begin{equation}
\label{23/07/2013(30)}
\begin{cases}
y(t) = e^{tA}\phi_0 + \displaystyle\int^t_0 e^{(t-s)A}Fy_sds + \int^t_0 e^{(t-s)A}f(s)ds,\hskip 20pt t\ge 0,\\
y(0)=\phi_0,\,\, y(\theta)=\phi_1(\theta),\,\,\theta\in [-r, 0],\,\,\,\Phi=(\phi_0, \phi_1)\in {\cal X},
\end{cases}
\end{equation}
where the delay term $F:\,C([-r, 0]; Z)\to X$ is a  bounded linear operator satisfying (\ref{14/10/13(1)}).

\begin{theorem}
\label{29/05/2013(1)} 
For any $T\ge 0$,  $f\in L^2([0, T]; X)$ and $\Phi=(\phi_0, \phi_1)\in {\cal X}$, there exists a unique solution $y(t) =y(t, \Phi)$ of (\ref{23/07/2013(30)}) such that 
\[
y\in L^2([0, T]; Z)\cap W^{1, 2}([0, T]; X)\subset C([0, T]; W).\]
 Moreover, there is a number $C_T>0$, depending only on $T$, such that 
\begin{equation}
\label{29/05/2013(3)}
\Big(\int^T_0 \|y(t)\|_Z^2dt + \int^T_0\Big\|\frac{dy(t)}{dt}\Big\|^2_{X}dt\Big) \le C_T\Big(\|\phi_0\|^2_W + \int^0_{-r}\|\phi_1(\theta)\|^2_Zd\theta + \int^T_0 \|f(t)\|^2_{X}dt\Big).
\end{equation}
\end{theorem}
\begin{proof} See \cite{Gdbkkes84(1)} or  \cite{jjsnht1993}.
\end{proof}

By Theorem \ref{29/05/2013(1)}, one can construct a family of  {\it fundamental solution\/}  $G(t): (-\infty, \infty)\to {\mathscr L}(W)$ for  (\ref{23/07/2013(30)}) with $f=0$  through
\begin{equation}
\label{25/05/06(1)}
G(t)x = \begin{cases}
y(t, \Phi),\hskip 15pt &t\ge 0,\\
0,\hskip 15pt &t< 0,
\end{cases}
\hskip 20pt \forall\, x\in W,
\end{equation} 
 where $\Phi=(x, 0).$ 
Let $A_i\in {\mathscr L}(Z, X)$, $i=1,\,2$, such that each $A_i$ maps ${\mathscr D}(A)$ into $X$ continuously.  In the sequel, we shall mainly focus on the following  form  of delays given by 
\begin{equation}
\label{11/08/2013(10)}
\eta(\theta) = -{\bf 1}_{(-\infty, -r]}(\theta)\alpha A_1 - \int^0_\theta \beta(\tau)A_2d\tau:\, Z\to X,\hskip 15pt \theta\in [-r, 0],  
\end{equation}
where  $\alpha\in {\mathbb R}$ and  $\beta\in L^2([-r, 0]; {\mathbb R})$. The main reason of this consideration is that it would allow us to have  a stronger regularity of $G(t)$ analogous to that for the analytic semigroup $e^{tA}$, $t\ge 0$. To this end, we impose further conditions on the kernel function $\beta$ in (\ref{11/08/2013(10)}):  suppose that   $\beta: [-r, 0]\to {\mathbb R}$ is an H\"older continuous function on $[-r, 0]$, i.e., there is a number $\rho\in (0, 1]$ such that 
\[
|\beta(t)-\beta(s)|\le C|t-s|^\rho \hskip 15pt \hbox{for any}\hskip 15pt t,\,\,s\in [-r, 0],\]
where $C>0$. Under this condition, we can solve (see \cite{jjsnht1993}) the equation (\ref{23/07/2013(30)}) in the following form on $X$,
\begin{equation}
\label{17/05/06(3)}
\begin{cases}
y(t) = e^{tA}\phi_0 + \displaystyle\int^t_0 e^{(t-s)A}\alpha A_1y(t-r)ds + \displaystyle\int^t_0\int^0_{-r}\beta(\theta)e^{(t-s)A}A_2y(t+\theta)d\theta ds\\
\hskip 150pt + \displaystyle\int^t_0 f(s)ds,\,\,\,\,\,\,t\in [0, T],\\
y(0) =\phi_0,\hskip 5pt y_0=\phi_1,\,\,\,\,\,\Phi=(\phi_0, \phi_1)\in  {\cal X},
\end{cases}
\end{equation}
or solve the corresponding operatoral equation
 \begin{equation}
\label{16/08/2013(1)}
\begin{cases}
G(t) = e^{tA} + \displaystyle\int^t_0 e^{(t-s)A}\alpha A_1G(t-r)ds + \displaystyle\int^t_0\int^0_{-r}\beta(\theta)e^{(t-s)A}A_2G(t+\theta)d\theta ds,\,\,\,t\in [0, T],\\
G(t) ={\rm O},\hskip 10pt t<0.
\end{cases}
\end{equation}
to get the fundamental solution $G(t)$ in the space ${\mathscr L}(X)$. 

\begin{proposition}
\label{14/08/2013(91)}
 {\rm (see \cite{jjsnht1993})}
The fundamental solution $G(t)$, $t\in {\mathbb R}$, of retarded type in (\ref{16/08/2013(1)}) is strongly continuous both in $X$ and $W$ such that $G(t):\, X\to  Z$ for each $t>0$ and satisfies
\[
\frac{d}{dt}G(t)x= AG(t)x + \alpha A_1G(t-r)x + \displaystyle\int^0_{-r} \beta(\theta) A_2 G(t+\theta)xd\theta,\hskip 20pt x\in X,\hskip 20pt t\ge 0,
\]
and
\[
\frac{d}{dt}G(t)x= G(t)Ax + G(t-r)\alpha A_1x + \displaystyle\int^0_{-r} \beta(\theta) G(t+\theta)A_2xd\theta,\hskip 20pt x\in Z,\hskip 20pt t\ge 0.
\]
\end{proposition}

\begin{corollary}
\label{14/08/2013(9199)}
Let $G^*(t)$ denote the adjoint operator of the fundamental solution $G(t)$, $t\in {\mathbb R}^1$. Then $G^*(t): Z^*\to X^*$ is strongly continuous  in $Z^*$ and satisfies
\[
\frac{d}{dt}G^*(t)x= G^*(t)A^*x + \alpha G^*(t-r) A^*_1x + \displaystyle\int^0_{-r} \beta(\theta) G^*(t+\theta)A^*_2xd\theta,\hskip 20pt x\in X^*,\hskip 20pt t\ge 0.
\]
and
\[
\frac{d}{dt}G^*(t)x= A^*G^*(t)x + \alpha A^*_1G^*(t-r)x + \displaystyle\int^0_{-r} \beta(\theta) A^*_2 G^*(t+\theta)xd\theta,\hskip 20pt x\in Z^*,\hskip 20pt t\ge 0.
\]
\end{corollary}

\begin{proposition}
For any $T\ge 0$, there exists a number $C_T>0$ such that for any  $f\in L^2([0, T]; X)$, 

\begin{equation}
\label{29/05/2013(2999)}
 \int^T_0 \Big\|\int^t_0 G(t-s)f(s)ds\Big\|^2_Zdt \le C_T\int^T_0 \|f(t)\|^2_{X}dt.
\end{equation}
\end{proposition}
\begin{proof}
Let $\Phi=(0, 0)$ in (\ref{29/05/2013(3)}), then we get the desired result (\ref{29/05/2013(2999)}) immediately.
\end{proof}

Let $\Phi\in {\cal X}$ and $y(t, \Phi)$ be the solution of the equation (\ref{17/05/06(3)}) with $f=0$. The segment process $y_t$ is given by $y_t(\Phi) = y(t+\theta; \Phi)$, $\theta\in [-r, 0]$. In association with $y$, we define the mapping ${\cal T}(t)$, $t\ge 0$,  of  (\ref{17/05/06(3)}) by 
\begin{equation}
\label{18/08/2013(70)}
{\cal T}(t)\Phi= (y(t; \Phi), y_t(\Phi)),\hskip 20pt t\ge 0,\,\,\,\,\Phi\in {\cal X}.
\end{equation}
Then it may be shown that ${\cal T}(t)$, $t\ge 0$, is a strongly continuous or $C_0$ semigroup on ${\cal X}$. Let ${\cal A}$ be the infinitesimal generator of ${\cal T}(t)$ or $e^{t{\cal A}}$, $t\ge 0$. The characterization of ${\cal A}$ is given by the following theorem.
\begin{theorem} {\rm (see \cite{Gdbkkes84(1)})} 
The operator ${\cal A}$ is described by 
\begin{equation}
\label{11/08/2013(20)}
\begin{split}
{\mathscr D}({\cal A}) = \Big\{\Phi=(\phi_0, \phi_1):\, &(\phi_0, \phi_1) \in {\mathscr D}(A)\times W^{1, 2}([-r, 0]; Z), \phi_0=\phi_1(0),\\
& A\phi_0 +\alpha A_1\phi_1(-r)  +  \displaystyle\int^0_{-r}\beta(\theta)A_2\phi_1(\theta)d\theta\in W\Big\},
\end{split}
\end{equation}
and for any $\Phi=(\phi_0, \phi_1)\in {\mathscr D}({\cal A})$, 
\begin{equation}
\label{20/08/2013(20)}
{\cal A}\Phi= \Big( A\phi_0 + \alpha A_1\phi_1(-r)  +  \displaystyle\int^0_{-r}\beta(\theta)A_2\phi_1(\theta)d\theta, \frac{d\phi_1(\theta)}{d\theta}\Big).
\end{equation}
\end{theorem}

For arbitrary $\lambda\in {\mathbb C}$, we define the characteristic operator $\Delta$ or $\Delta(\lambda)$ of (\ref{17/05/06(3)}) by
\begin{equation}
\label{20/08/2013(70)}
\begin{split}
\Delta(\lambda) = \lambda -A -\alpha A_1e^{-\lambda r} -\int^0_{-r} \beta(\theta)  e^{\lambda \theta}A_2d\theta.
\end{split}
\end{equation}
Clearly, $\Delta(\lambda)\in {\mathscr L}({\mathscr D}(A), X)$ for each $\lambda\in {\mathbb C}$. We also define the {resolvent} and {spectrum} sets for $\Delta(\lambda)$, respectively, by
\[
\rho(\Delta) = \{\lambda:\,\,\Delta(\lambda)\,\,\hbox{is bijective}\}\hskip 15pt\hbox{and}\hskip 15pt \sigma(\Delta) ={\mathbb C}\,\backslash\, \rho(\Delta).\]
Then it is easy to see, by the well-known open mapping theorem, that for each $\lambda\in \rho(\Delta)$, the inverse $\Delta(\lambda)^{-1}$ exists and belongs to ${\mathscr L}(X)$.

\begin{definition} \rm We introduce the following notations of continuous, residual and point spectra of $\Delta$:

$\sigma_C(\Delta) =\{\lambda\in {\mathbb C}:\, \Delta(\lambda)\,\,\hbox {is injective,}\,\,\overline{\Delta(\lambda){\mathscr D}(A)} = X, \,\Delta(\lambda)^{-1}\,\,\hbox{is unbounded on}\,\,X\},$

$\sigma_R(\Delta) =\{\lambda\in {\mathbb C}:\, \Delta(\lambda)\,\,\hbox {is injective,}\,\,\overline{\Delta(\lambda){\mathscr D}(A)} \not= X\},$

$\sigma_P(\Delta) =\{\lambda\in {\mathbb C}:\, \Delta(\lambda)\,\,\hbox {is not injective}\}.$
\end{definition}
\noindent By definition, it is clear that ${\mathbb C} = \rho(\Delta)\cup \sigma_C(\Delta)\cup \sigma_R(\Delta)\cup\sigma_P(\Delta).$
 The following result provides a useful tool to establish the spectrum relations between $\Delta$ and  the generator ${\cal A}$.   

\begin{proposition}
\label{14/08/2013(70)}
Let $\lambda\in {\mathbb C}$ and $\Psi=(\psi_0, \psi_1)\in {\cal X}$. If $\Phi=(\phi_1(0), \phi_1)\in {\mathscr D}({\cal A)}$ satisfies 
\begin{equation}
\label{28/07/2013(2)}
\lambda\Phi -{\cal A}\Phi = \Psi,
\end{equation}
then we have 
\begin{equation}
\label{28/07/2013(5)}
\phi_1(\theta) = e^{\lambda\theta}\phi_1(0) + \int^0_{\theta}e^{\lambda(\theta-\tau)}\psi_1(\tau)d\tau,\hskip 20pt -r\le \theta\le 0,
\end{equation}
and, letting $\phi_0=\phi_1(0)$, there is
\begin{equation}
\label{28/07/2013(1)}
\Delta(\lambda) \phi_1(0) =  \int^0_{-r} e^{\lambda(-r-\tau)}\alpha A_1\psi_1(\tau)d\tau + \int^0_{-r} \beta(\theta) \int^0_{\theta} e^{\lambda(\theta-\tau)}A_2\psi_1(\tau)d\tau d\theta+\psi_0.
\end{equation}
Conversely, if $ \phi_0\in {\mathscr D}(A)$ satisfies the equation (\ref{28/07/2013(1)}), and letting $\phi_1(0)=\phi_0$, 
\begin{equation}
\label{28/07/2013(20)}
\phi_1(\theta) = e^{\lambda\theta} \phi_1(0) +  \int^0_{\theta}e^{\lambda(\theta-\tau)}\psi_1(\tau)d\tau,\hskip 20pt -r\le \theta\le 0,
\end{equation}
then we have that $\phi_1\in W^{1,2}([-r, 0]; Z)$, $\Phi=(\phi_1(0), \phi_1)\in {\mathscr D}({\cal A})$ and $\Phi$ satisfies (\ref{28/07/2013(2)}).
\end{proposition}
\begin{proof} See Appendix.
\end{proof}

As usual we denote by $\rho({\cal A})$ the resolvent set of ${\cal A}$,  $\sigma({\cal A})$ the spectrum of ${\cal A}$ and by $\sigma_P({\cal A})$, $\sigma_C({\cal A})$, $\sigma_R({\cal A})$ the point, continuous and residual spectra of ${\cal A}$, respectively. By virtue of Proposition \ref{14/08/2013(70)}, we can establish the following results on the relationship between three kinds of spectrum for ${\cal A}$ and the corresponding  $\Delta$.
\begin{theorem} 
\label{20/08/2013(80)}
{\rm (see \cite{snht96})} For the operators $\Delta$ and ${\cal A}$ of (\ref{17/05/06(3)}), the  following inclusions and equalities hold:
\begin{equation}
  \sigma_P({\cal A})=\sigma_P(\Delta), 
\end{equation}
\begin{equation}
\sigma_R({\cal A}) = \sigma_R(\Delta),
\end{equation}
\begin{equation}
\sigma_C({\cal A}) \subset \sigma_C(\Delta)\subset \sigma_C({\cal A})\cup  \rho({\cal A}).
\end{equation}
\end{theorem}

\section{Variation of Constants Formula}

In the sequel, we assume that $Z$, $X$ are real separable Hilbert spaces and $W=W^*$ according to the well-known Riesz representation theorem. 
Let $L^2_{{\mathscr F}_0}(\Omega; {\cal X})$ denote the space of all ${\cal X}$-valued mappings  $\Phi(\omega)= (\phi_0(\omega), \phi_1(\cdot, \omega))$ defined on some probability space $(\Omega, {\mathscr F}, \{{\mathscr F}_t\}_{t\ge 0}, {\mathbb P})$ such that both $\phi_0$ and $\phi_1(\theta)$ are ${\mathscr F}_0$-measurable for any $\theta\in [-r, 0]$ and satisfy
\[
{\mathbb E}\|\Phi\|^2_{\cal X} ={\mathbb E}\|\phi_0\|^2_W + {\mathbb E}\|\phi_1\|^2_{L^2_r}<\infty.\]
As mentioned before, we shall be concerned about  the following stochastic retarded evolution equation  on the Hilbert space $X,$
\begin{equation}
\label{13/08/2013(1)}
\begin{cases}
dy(t) = Ay(t)dt + \alpha A_1y(t-r)dt + \displaystyle\int^0_{-r} \beta(\theta)A_2y(t+\theta)d\theta dt +f(t)dB(t),\,\,\,\,\,\,t\in [0, T],\\
y(0) =\phi_0,\hskip 5pt y_0=\phi_1,\,\,\,\,\,\Phi=(\phi_0, \phi_1)\in L^2_{{\mathscr F}_0}(\Omega; {\cal X}),
\end{cases}
\end{equation}
where $f\in L^2(\Omega\times [0, T]; X)$ and $B$ is a real-valued ${\mathscr F}_t$-Brownian motion  on $(\Omega, {\mathscr F}, {\mathbb P})$. Here $A$, $A_1$, $A_2$ and $\alpha$, $\beta$ all are  given as in  (\ref{17/05/06(3)}).
We may establish the following proposition which is crucial for the variation of constants formula of the solutions for (\ref{13/08/2013(1)}).
\begin{proposition}
\label{16/08/2013(10)}
Let $G(\cdot)$ be the fundamental solution of (\ref{13/08/2013(1)}). Then the  process $v(t) := \int^t_0 G(t-s)f(s)dB(s)$ constitutes a solution of the equation (\ref{13/08/2013(1)}) with $\phi_0=0$, $\phi_1\equiv 0$ and moreover 
\begin{equation}
\label{13/08/2013(6}
v\in L^2([0, T]\times \Omega; Z)\cap L^2(\Omega; C([0, T]; W))\hskip 15pt \hbox{for any}\hskip 15pt T\ge 0.
\end{equation}
Hence $v(t)$, $t\in [0, T]$, gives the unique solution of (\ref{13/08/2013(1)}) with zero initial data.
\end{proposition}
\begin{proof}
We split our proofs into two steps as follows.

{\it Step 1.} Let $f\in L^2(\Omega\times [0, T]; X)$. We first show that $v(t) := \int^t_0 G(t-s)f(s)ds$ is a solution of the equation (\ref{17/05/06(3)}) with $\phi_0=0$ and $\phi_1\equiv 0$. To this end, first note that $G(t)$ is strongly continuous in $X$ and $v(t)$ makes sense as a Bochner integral in $X$ for each $t>0$. Since $G(t):\, X\to Z$ for each $t>0$, we have that
\begin{equation}
\label{13/08/2013(2)}
\begin{split}
v(t) &= \int^t_0 e^{(t-s)A}f(s)ds + \int^t_0\Big(\int^{t-s}_0 \alpha e^{(t-s-\tau)A}A_1G(\tau-r)d\tau\Big)f(s)ds\\
&\,\,\,\,\,\, + \int^t_0\Big(\int^{t-s}_0 e^{(t-s-\tau)A}\int^0_{-r} \beta(\theta)A_2G(\tau+\theta)d\theta d\tau\Big)f(s)ds\\
&=: I_1 +I_2+I_3\in Z,\hskip 20pt t\in [0, T].
\end{split}
\end{equation}
Using Fubini's theorem  and noting that $G(t)={\rm O}$ for $t<0$, we transform the integrals $I_2$, $I_3$ in (\ref{13/08/2013(2)}) into 
\begin{equation}
\label{13/08/2013(3)}
\begin{split}
I_2 &=  \alpha\int^t_0\int^{t}_s e^{(t-\tau)A}A_1G(\tau-s-r)d\tau f(s)ds\\
&= \alpha\int^t_0 e^{(t-\tau)A}A_1\Big(\int^0_{\tau} G(\tau-s-r)f(s)ds\Big)d\tau,
\end{split}
\end{equation}
and in a similar way, it is easy to see that 
\begin{equation}
\label{13/08/2013(4)}
\begin{split}
I_3 &=  \int^t_0\int^{t}_s e^{(t-\tau)A}\int^0_{-r}\beta(\tau) A_2 G(\tau-s+\theta)f(s)d\theta d\tau ds\\
&= \int^t_0 e^{(t-\tau)A}\Big(\int^{0}_{-r}\beta(\theta)A_2\int^\tau_0  G(s+\theta-\tau)f(\tau)d\tau d\theta \Big)ds.
\end{split}
\end{equation}
Hence combining (\ref{13/08/2013(2)})--(\ref{13/08/2013(4)}), we can verify immediately that $v(t)$ satisfies the equation (\ref{17/05/06(3)}) with $\phi_0=0$ and $\phi_1\equiv 0$. Furthermore, by Theorem \ref{29/05/2013(1)} it is easy to see that $v(t)$ is the unique solution of  (\ref{17/05/06(3)}) and moreover $v(t)$ satisfies (\ref{13/08/2013(6}).

{\it Step 2.} We first show that $v(t) := \int^t_0 G(t-s)f(s)dB(s)$ is a solution of the equation (\ref{13/08/2013(1)}) with $\phi_0=0$ and $\phi_1\equiv 0$ when  $f\in L^2(\Omega\times [0, T]; Z)$.  Define 
\begin{equation}
\label{15/08/2013(100)}
M(t) =
\begin{cases}
 \displaystyle\int^t_0 f(s)dB(s)\hskip 20pt \hbox{ for}&\hskip 15pt t\ge 0,\\
0\hskip 88pt \hbox{for}&\hskip 15pt t\in (-r, 0].
\end{cases}
\end{equation}
Then it is easy to see that $M \in L^2(\Omega\times [0, T]; Z)$.  Now let us  consider the following stochastic system with time delay,
\begin{equation}
\label{13/08/2013(20)}
\begin{cases}
y(t)=  \displaystyle\int^t_0  Ay(s)ds + \int^t_0 \alpha A_1 y(s-r)ds + \int^t_0 \int^0_{-r} \beta(\theta) A_2 y(s+\theta)d\theta ds + M(t),\,\,\, t\in [0, T],\\
y(0) = 0,\,\,\,\,y(\theta) =0, \,\,\,\,\,\theta\in [-r, 0].
\end{cases}
\end{equation} 
Let $z(t)=y(t)-M(t)$, $t\in [0, T]$. Then it is immediate to see that for any $t\in [0, T],$
\begin{equation}
\label{13/08/2013(21)}
\begin{split}
z(t) &= \int^t_0 A(z(s)+M(s))ds + \int^t_0 \alpha A_1( z(s-r) + M(s-r))ds\\
&\,\,\,\,\,\,\, + \int^t_0 \int^0_{-r} \beta(\theta) A_2 (z(s+\theta) + M(s+\theta))d\theta ds\\
&= \int^t_0 Az(s)ds + \int^t_0 \alpha A_1z(s-r)ds + \int^t_0 \int^0_{-r} \beta(\theta) A_2z(s+\theta)d\theta ds\\
&\,\,\,\,\,\, +  \int^t_0 \Big[AM(s) + \alpha A_1M(s-r) + \int^0_{-r} \beta(\theta) A_2M(s+\theta)d\theta\Big] ds.
\end{split}
\end{equation}
According to Step 1, we have $z\in  L^2(\Omega\times [0, T]; Z)\cap L^2(\Omega; C([0, T]; W))$, and further we may obtain the explicit form of the solution of (\ref{13/08/2013(21)}) as
\begin{equation}
\label{14/08/2013(90)}
\begin{split}
z(t) &= y(t)-M(t)\\
&= \int^t_0 G(t-s) \Big[AM(s) + \alpha A_1M(s-r) + \int^0_{-r} \beta(\theta) A_2M(s+\theta)d\theta\Big]ds
\end{split}
\end{equation}
for all $t\in [0, T].$
On the other hand, we may derive by using (\ref{15/08/2013(100)}), Fubini's theorem and Proposition \ref{14/08/2013(91)}  that for all $t\in [0, T]$,
\[
\begin{split}
 \int^t_0 G(t-s) &A\int^s_0 f(u)dB(u) ds+  \int^t_r G(t-s) \alpha A_1\int^{s-r}_0 f(u)dB(u)ds\\
& + \int^t_0 G(t-s)\int^0_{-r} \beta(\theta) A_2\int^{s+\theta}_0 f(u)dB(u)d\theta ds\\
&= \int^t_0 \Big[\int^t_u G(t-s)Af(u)ds + \int^{t-r}_u G(t-s)\alpha A_1 f(u)ds\\
&\,\,\,\,\,\,\,\, + \int^0_{-r} \beta(\theta) G(t-s)A_2\int^{s+\theta}_0 f(u)d\theta ds\Big]dB(u)\\
&= - \int^t_0\int^t_u \frac{dG(t-s)}{ds}f(u)dsdB(u)= \int^t_0 (G(t-u) -I)f(u)dB(u),
\end{split}
\]
which, in addition to  (\ref{14/08/2013(90)}), immediately implies that 
\[
y(t) = \int^t_0 G(t-u)f(u)dB(u),\hskip 20pt t\in [0, T],\]
and $y(t) = z(t)+M(t) \in L^2(\Omega\times [0, T]; Z)\cap L^2(\Omega; C([0, T]; W))$.

Last, the general result can be easily obtained by choosing a sequence $\{f_n\}\in L^2(\Omega\times [0, T]; Z)$ such that $f_n\to f$ in $L^2(\Omega\times [0, T]; X)$ and passing on a limit procedure. The proof is thus complete.
\end{proof}

For each $t>0$, we introduce the operator-valued function $U_t(\cdot)$ defined by 
\[
U_t(\theta) = \alpha G(t-\theta-r) A_1 + \int^\theta_{-r} \beta(\tau)G(t-\theta+\tau) A_2d\tau,\hskip 20pt \theta\in [-r, 0].\]
Let $T>0$ and we consider in association with $U_t(\cdot)$ a linear operator $U:\, L^2([-r, 0]; Z)\to L^2([0, T]; Z)$ defined by
\[
(U\varphi)(t) = \int^0_{-r} U_t(\theta)\varphi(\theta)d\theta,\hskip 15pt t\in [0, T],\hskip 15pt \varphi\in L^2([-r, 0]; Z).\]
We may see that $U$ is into and bounded. To show this, it is useful to introduce the  structure operator $S:\, L^2([-r, 0]; Z)\to L^2([-r, 0]; X)$ given by
\begin{equation}
\label{01/02/2012(1)}
\begin{split}
[S\varphi](\theta) = \alpha A_1\varphi(-r-\theta) +\int^\theta_{-r} \beta(\tau)A_2\varphi(\tau-\theta)d\tau,\,\,\,\,\,\theta\in [-r, 0],\,\,\,\,a.e. 
\end{split}
\end{equation}
for all $\varphi\in L^2([-r, 0]; Z)$. By using Cauchy-Schwartz's inequality, we have for any $\varphi\in L^2([-r, 0]; Z)$ that
\begin{equation}
\label{01/02/2012(2)}
\begin{split}
\int^0_{-r} \|S\varphi(\theta)\|^2_{X}d\theta &= \int^0_{-r} \|\alpha A_1\varphi(-r-\theta) + \int^\theta_{-r} \beta(\tau)A_2\varphi(\tau-\theta)d\tau\Big\|^2_{X}d\theta\\
&\le 2\int^0_{-r} \|\alpha A_1\varphi(-r-\theta)\|^2_{X}d\theta +2\int^0_{-r} \Big\|\int^\theta_{-r}\beta(\tau)A_2\varphi(\tau-\theta)d\tau\Big\|^2_{X}d\theta\\
&\le 2\Big(\alpha^2\|A_1\|^2_{{\mathscr L}(Z, X)} + r\|A_2\|^2_{{\mathscr L}(Z, X)}\|\beta\|^2_{L^2([-r, 0]; {\mathbb R})}\Big)\int^0_{-r} \|\varphi(\theta)\|^2_Zd\theta.
\end{split}
\end{equation}
Hence, we see that $S$ is into and bounded. Further, we define the structure operator ${\cal S}:\, {\cal X}\to W\times L^2([-r, 0]; X)$ by 
\begin{equation}
\label{21/11/13(95)}
{\cal S}\Phi = (\phi_0, S\phi_1),\hskip 20pt \forall \,\Phi=(\phi_0, \phi_1)\in {\cal X}.
\end{equation}
It is straightforward to see that ${\cal S}$ is linear and bounded.
In terms of $S$, we can further derive by  Fubini's theorem  that for $t\in [0, T],$
\[
(U\varphi)(t) = \int^0_{-r}G(t+\theta)[S\varphi](\theta)d\theta = \int^t_0 G(t-\theta)\bar\varphi(\theta)d\theta\]
where $\bar\varphi(\cdot) ={\bf 1}_{[0, r]}[S\varphi](-\cdot)\in L^2([0, T]; X)$. By  Step 1 in the proofs of Proposition \ref{16/08/2013(10)}, it is easy to see that $U$ is into and bounded.
\begin{theorem}
Let $T>0$, $\Phi=(\phi_0, \phi_1)\in L^2_{{\mathscr F}_0}(\Omega; {\cal X})$ and $f\in L^2(\Omega\times [0, T]; X)$, the solution $y$ of  (\ref{13/08/2013(1)})  is represented by 
\begin{equation}
\label{16/08/2013(50)}
\begin{split}
y(t, \Phi) &= G(t)\phi_0 + \int^0_{-r} U_t(\theta)\phi_1(\theta)d\theta + \int^t_0 G(t-s)f(s)dB(s)\\
&= G(t)\phi_0 + \int^0_{-r} G(t+\theta)(S\phi_1)(\theta)d\theta + \int^t_0 G(t-s)f(s)dB(s),\hskip 20pt t\in [0, T].
\end{split}
\end{equation}
\end{theorem}
 \begin{proof}
By the uniqueness of solutions in the class $L^2(\Omega\times [0, T]; Z)\cap L^2(\Omega; C([0, T]; W))$, it is obvious from the definition of fundamental solution $G$  and Proposition \ref{16/08/2013(10)} that
\begin{equation}
\label{18/08/2013(1)}
y(t, \Phi) = G(t)\phi_0 + \int^t_0 G(t-s)f(s)dB(s)\hskip 15pt \hbox{for}\hskip 15pt \Phi=(\phi_0, 0),\,\,\,\,\phi_0\in W.
\end{equation}
For $f=0$, $\phi_0=0$ and $\phi_1\in L^2_{{\mathscr F}_0}([-r, 0]; Z)$, it can be shown  shown as in \cite{jjsnht1993} that 
\begin{equation}
\label{18/08/2013(2)}
y(t, \Phi) = \int^0_{-r} U_t(\theta)\phi_1(\theta)d\theta,\hskip 20pt \Phi=(0, \phi_1)\in L^2_{{\mathscr F}_0}(\Omega; {\cal X}).
\end{equation}
Combining (\ref{18/08/2013(1)}) and (\ref{18/08/2013(2)}), we may easily show  the formula (\ref{16/08/2013(50)}).
\end{proof}

\section{Stationary Solution}

We consider  the  system (\ref{13/08/2013(1)}) with deterministic initial data $\Phi\in {\cal X}$ and $f(\cdot)\equiv f\in X$.

\begin{definition}
\rm A solution $y=\{y(t); t\ge -r\}$  of  (\ref{13/08/2013(1)})
is called  {\it strongly stationary}, or simply {\it stationary},  if for any $h_1,\,\cdots, h_n\in W$,
\begin{equation}
\label{24/02/09(1)}
{\mathbb E}\Big[\exp\Big(i\sum^n_{k=1}\langle y(t_k+s), h_k\rangle_W\Big)\Big] =  {\mathbb E}\Big[\exp\Big(i\sum^n_{k=1}\langle y(t_k), h_k\rangle_W\Big)\Big],
\end{equation}
for all $s\ge 0$, $t_k\ge -r$, $k=1,\cdots, n$.
We say that (\ref{13/08/2013(1)}) has a stationary  solution $y$ if there exists an initial $\Phi=(\phi_0, \phi_1)\in {\cal X}$ such that $y(t, \Phi)$, $t\ge 0$, is a stationary solution of (\ref{13/08/2013(1)}) with $y(0)=\phi_0$, $y_0=\phi_1$.
A stationary solution is said to be {\it uniquely determined\/} if any two stationary solutions of   (\ref{13/08/2013(1)}) have the same finite dimensional distributions.
\end{definition}

\begin{theorem} 
\label{16/08/2013(100)} 
Suppose that  the $C_0$-semigroup  $e^{t{\cal A}}$, $t\ge 0$, defined in (\ref{18/08/2013(70)}) is exponentially stable, i.e., there exist constants $M\ge 1$ and $\mu>0$ such that 
\begin{equation}
\label{23/09/07(4)}
\|e^{t{\cal A}}\|\le Me^{-\mu t}\,\,\,\,\hbox{for all}\,\,\,\,\,t\ge 0,
\end{equation}
 then there exists a unique stationary solution of  (\ref{13/08/2013(1)}). This stationary solution is a zero mean Gaussian process with the covariance operator $K(\cdot)$ given on $Z$ by
\begin{equation}
\label{10/09/07(1)}
\begin{split}
 K(t) = \int^\infty_0 (G(t+s)f)\otimes (G(s)f)ds,\hskip 15pt t\ge -r.
\end{split}
\end{equation}
Here for  $a,\,b\in W,$ $(a\otimes b)v := a\langle b, v\rangle_W\in W$ for any $v\in W$.
Moreover,  for any $x\in Z$,  $K(t)x\in Z$,  $K(t)x$ is strongly differentiable and
\begin{equation}
\label{14/08/07(1)}
\frac{dK(t)}{dt}x= AK(t)x + \alpha A_1K(t-r) x + \int^0_{-r} \beta(\theta)A_2K(t+\theta)xd\theta,\hskip 20pt t\ge 0.
\end{equation}
\end{theorem}
\begin{proof} 
For any $x\in W$, let $\Phi=(x, 0)$. For such an initial $\Phi\in {\cal X}$, $G(t)x= y(t, \Phi)$, the  solution of  (\ref{17/05/06(3)}) with $f\equiv 0$, and we thus have by virtue of (\ref{23/09/07(4)})
 that
\begin{equation}
\label{20/08/2013(1)}
\|G(t)x\|_W\le \|e^{t{\cal A}}\Phi\|_{\cal X} \le Me^{-\mu t}\|\Phi\|_{\cal X}= Me^{-\mu t}\|x\|_{W}\hskip 10pt \hbox{for all}\hskip 10pt t\ge 0,
\end{equation}
where $M\ge 1$ and $\mu>0$. Next, we split the remaining proofs into  several steps.

{\it Step 1}.  Let $B_1(t)$ and $B_2(t)$, $t\ge 0$, be two independent real-valued Brownian motion. We first extend them to obtain a two-sided Brownian motion on the whole time axis ${\mathbb R}^1$ by 
\begin{equation}
\label{23/02/09(1)}
B(t) =
\begin{cases}
B_1(t),\hskip 42pt t\ge 0,\\
-B_2(-t),\hskip 25pt t<0,
\end{cases}
\end{equation}
and  for $t\ge -r$, let
\begin{equation}
\label{23/02/09(2)}
\begin{split}
U(t) :=  \int^t_{-\infty} G(t-s)f dB(s)
\end{split}
\end{equation}
By virtue of (\ref{20/08/2013(1)}), it is easy to see that the process $U(t)\in Z\subset W$, $t>0$, in (\ref{23/02/09(2)}) is well-defined.  Also it is immediate  that ${\mathbb E}U(t)=0$ and the process $U(t)$, $t\ge 0$, is Gaussian. Moreover, let $-r\le t_1<\cdots < t_n$, we have for any $h_1,\,\cdots,\,h_n\in W$ that
 \begin{equation}
\label{23/02/09(3)}
\begin{split}
{\mathbb E}\exp\Big(i\sum^n_{k=1}&\langle h_k, U(t_k)\rangle_W\Big)\\
 =&\,\, \exp\Big\{-\frac{1}{2}\Big[\int^\infty_{t_n} \sum^n_{i,\,j=1}\langle (G(t_i+s-t_n)f)\otimes (G(t_j+s-t_n)f)h_i, h_j\rangle_W ds\\
&\,\, + {\bf 1}_{\{t_n>0\}}\int^{t_n}_0  \sum_{t_i,\,t_j>0}\langle (G(t_i +s-t_n)f)\otimes (G(t_j +s-t_n)f)h_i, h_j\rangle_W ds\Big]\Big\}\\
 =&\,\, \exp\Big\{-\frac{1}{2}\Big[\int^\infty_{0} \sum^n_{i,\,j=1}\langle (G(t_i-t_j+s)f)\otimes G(s)f)h_i, h_j\rangle_W ds\Big]\Big\}.
\end{split}
\end{equation}
Thus, the process $U$ is stationary in the sense of (\ref{24/02/09(1)}). Moreover, from (\ref{23/02/09(3)}) we get that this stationary solution is a zero mean Gaussian process with covariance operator given by (\ref{10/09/07(1)}).

{\it Step 2.} We show that $U(t)$, $t\ge -r$, in (\ref{23/02/09(2)}) is a  solution of (\ref{13/08/2013(1)}). To this end,  let $v\in X^*$ and by using the  stochastic Fubini's theorem, Corollary \ref{14/08/2013(9199)}  and the fact that $G(t)={\rm O}$ for $t<0$,
 we have for any $t\ge 0$ that
\[
\label{23/02/09(10)}
\begin{split}
&\int^t_0 \langle U(s), A^*v\rangle_{Z, Z^*} ds + \Big\langle \int^0_{-\infty} G(-u)fdB(u), v\Big\rangle_{X, X^*}\\
& = \int^t_0 \Big\langle \int^s_{-\infty} G(s-u)fdB(u), A^*v\Big\rangle_{Z, Z^*}ds +  \Big\langle\int^0_{-\infty} G(-u)fdB(u), v\Big\rangle_{X, X^*}\\
&= \int^0_{-\infty}\Big\langle fdB(u), \int^t_0 \frac{d}{ds}G^*(s-u)vds\Big\rangle_{X, X^*} +  \int^t_{0}\Big\langle fdB(u), \int^t_u \frac{d}{ds}G^*(s-u)vds\Big\rangle_{X, X^*}\\
&\,\,\,\,+ \Big\langle \int^0_{-\infty} G(-u)fdB(u), v\Big\rangle_{X, X^*}\\
&\,\,\,\, - \int^t_0 \Big\langle \int^s_{-\infty} \alpha A_1G(s-u-r)fdB(u) + \int^s_{-\infty} \int^0_{-r}\beta(\theta)A_2G(s-u+\theta)fd\theta dB(u), v\Big\rangle_{X, X^*}ds\\
& = \int^t_{-\infty}\Big\langle fdB(u), G^*(t-u)v\Big\rangle_{X, X^*}  - \int^0_{-\infty} \Big\langle fdB(u), G^*(-u)v\Big\rangle_{X, X^*} -\int^t_0 \langle fdB(u), v\rangle_{X, X^*}\\
&\,\,\,\, + \Big\langle \int^0_{-\infty} G(-u)fdB(u), v\Big\rangle_{X, X^*} - \int^t_0 \Big\langle \int^{s-r}_{-\infty} \alpha A_1G(s-r-u)fdB(u)\\
&\,\,\,\,+ \int^0_{-r}\int^{s+\theta}_{-\infty} \beta(\theta)A_2G(s-u+\theta)fdB(u)d\theta, v\Big\rangle_{X, X^*}ds\\
&= \Big\langle U(t) -\int^t_0 fdB(u), v\Big\rangle_{X, X^*} - \int^t_0\Big\langle  \alpha A_1U(s-r) + \int^0_{-r} \beta(\theta)A_2U(s+\theta)d\theta, v\Big\rangle_{X, X^*}ds.  
\end{split}
\]
 Since $v\in X^*$ is arbitrary, we get that $U(t)$, $t\ge 0$, in (\ref{23/02/09(2)}) is a  solution of (\ref{13/08/2013(1)}). 

{\it Step 3.}
To show the remainder of the theorem, we use  Corollary \ref{14/08/2013(9199)} to derive that for any $v\in X^*$,
\[
\begin{split}
\Big[\alpha A_1&K(t-r) + \int^0_{-r} \beta(\theta)A_2K(t+\theta)d\theta\Big]^*v\\
&=   \int^\infty_0 \alpha (G(s)^*f)\otimes (G(t+s-r)^*f)A_1^*vds\\
&\,\,\,\,\,\,\, + \int^0_{-r}\beta(\theta)\int^\infty_0 (G(s)^*f)\otimes (G(t+s-r)^*f)A^*_2vds d\theta\\
&= \int^\infty_0 \Big((G^*(s)f)\otimes \frac{d}{dt}(G^*(s+t)f)\Big)vds -  \int^\infty_0 \Big((G^*(s)f)\otimes (G^*(s+t)f)A^*\Big)vds\\
&= \frac{d}{dt}K^*(t)v - K^*(t)A^*v.
\end{split}
\]
From this, one can easily get that the derivative $(dK(t)/dt)x$ exists for any $x\in Z$ and moreover the equality  (\ref{14/08/07(1)}) holds true. The proof is thus complete.
\end{proof}

In the sequel we shall use Theorem \ref{16/08/2013(100)} to find stationary solutions for the fundamental model equation (\ref{13/08/2013(1)}). 
 It is well known that when the solution semigroup $e^{t{\cal A}}$, $t\ge 0$, of (\ref{13/08/2013(1)}) and its infinitesimal generator ${\cal A}$ defined in (\ref{18/08/2013(70)}) and (\ref{20/08/2013(20)}) satisfy the spectral mapping theorem, then 
\begin{equation}
\label{20/08/2013(30)}
\sup\{\hbox{Re}\,\lambda:\,\lambda\in \sigma({\cal A})\} = \inf\{\mu\in {\mathbb R}:\,\|e^{t{\cal A}}\|\le Me^{\mu t}\,\,\hbox{for some}\,\, M>0\}.
\end{equation}
In other words, the stability properties of the semigroup $e^{t{\cal A}}$, $t\ge 0$, can be obtained by the location of the spectrum of $A$. For instance, this can be done when the  semigroup, $e^{t{\cal A}}$, $t\ge 0$, is compact. In  \cite{kl08(2), kl11}, it is shown that if $A$ generates a compact semigroup and $A_1$, $A_2$ both are bounded, then the semigroup $e^{t{\cal A}}$, $t\ge 0$, is eventually compact. In this case, the relation  (\ref{20/08/2013(30)}) could be used to obtain stationary solutions of (\ref{13/08/2013(1)}). 

When $A_1$, $A_2$ are unbounded, the situation becomes quite complicated.  For instance, let us consider Example 1.1 where $A$ generates a compact semigroup with $Z={\mathscr D}(A)$,  $A_1\in {\mathscr L}({\mathscr D}(A), X)$ and $A_2=0$,  it was shown  that the associated solution semigroup   $e^{t{\cal A}}$, $t\ge 0$,  in (\ref{13/08/2013(1)}) is generally  not compact (see \cite{Gdbkkes85(2)}) or even not eventually norm continuous (see \cite{jj1991}). On the other hand, for Example 1.2 with  $A_1=0$ and $A_2\in {\mathscr L}(V, V^*)$, the solution semigroup $e^{t{\cal A}}$, $t\ge 0$, in (\ref{13/08/2013(1)}) is generally not compact, although  it could be eventually norm continuous (see \cite{jj1991}). Due to this complexity, it is necessary for us to find stationary solutions for the stochastic system (\ref{13/08/2013(1)}) by handling  the discrete  and distributed delays  separately.

\section{Unbounded Delay Operators}

 We first state some results about the following deterministic equation 
\begin{equation}
\label{11/08/2013(30)}
\begin{cases}
{dy(t)}/{dt} =Ay(t)  + \alpha A y(t-r) +\displaystyle\int^0_{-r} \beta(\theta)A y(t+\theta)d\theta,\,\,\,\,t\ge 0,\\
y(0)=\phi_0,\,\,\,\,y_0=\phi_1,\,\,\,\Phi=(\phi_0, \phi_1)\in {\cal X},
\end{cases}
\end{equation} 
where $\alpha\in {\mathbb R}$ and $\beta\in L^1([-r, 0]; {\mathbb R})$. In this case, the characteristic operator $\Delta$  defined in (\ref{20/08/2013(70)}) is given by $\Delta(\lambda)x = \lambda x - n(\lambda)Ax$ for each $\lambda\in {\mathbb C}$, $x\in {\mathscr D}(A),$
where 
\begin{equation}
\label{23/08/13(21)}
n(\lambda) = 1 + \alpha e^{-\lambda r} + \int^0_{-r} \beta(\theta) e^{\lambda\theta}d\theta,\hskip 15pt \lambda\in {\mathbb C}.
\end{equation}
In addition, we define  
\begin{equation}
\label{01/09/13(10)}
\begin{cases}
\Gamma_C = \{\lambda\in {\mathbb C}:\, n(\lambda)\not = 0,\,\lambda n(\lambda)^{-1}\in \sigma_C(A)\},\\
\Gamma_R = \{\lambda\in {\mathbb C}:\, n(\lambda)\not = 0,\,\lambda n(\lambda)^{-1}\in \sigma_R(A)\},\\
\Gamma_P = \{\lambda\in {\mathbb C}:\, n(\lambda)\not = 0,\,\lambda n(\lambda)^{-1}\in \sigma_P(A)\},\\
\Gamma_0 = \{\lambda\in {\mathbb C}:\, \lambda\not = 0,\,n(\lambda)=0\},\\
\Gamma_1 = \{\lambda\in {\mathbb C}:\, n(\lambda)\not=0,\,\lambda n(\lambda)^{-1}\in \sigma(A)\}.\\
\end{cases}
\end{equation}

\begin{proposition}
\label{22/08/13(2)}
 {\rm (See \cite{snht96})} For the characteristic operator $\Delta$ and the associated generator ${\cal A}$ of the equation (\ref{11/08/2013(30)}), it is true that 

(i) $\Gamma_0\subset \sigma_C({\cal A}) \subset \sigma_C(\Delta) =\Gamma_C\cup \Gamma_0;$

(ii) $\sigma_R({\cal A}) = \sigma_R(\Delta) = \Gamma_R;$

(iii) 
\[
\sigma_P({\cal A}) = \sigma_P(\Delta) =
\begin{cases}
\Gamma_P\hskip 40pt &\hbox{if}\hskip 20pt 1 + \alpha + \displaystyle\int^0_{-r} \beta(\theta)d\theta\not= 0,\\
\Gamma_P\cup \{0\}\hskip 40pt &\hbox{if}\hskip 20pt 1 + \alpha + \displaystyle\int^0_{-r} \beta(\theta)d\theta= 0.
\end{cases}
\]
\end{proposition}

\subsection{Distributed Delay}

Now we pass on to consider the equation (\ref{13/08/2013(1)}) with  $A_1=0$, $A_2=A$ and $f(\cdot)=f\in W$, i.e., 
\begin{equation}
\label{11/08/2013(30070)}
\begin{cases}
\displaystyle\frac{dy(t)}{dt} =Ay(t)  + \displaystyle\int^0_{-r} \beta(\theta)A y(t+\theta)d\theta + f \dot B(t),\,\,\,\,t\ge 0,\\
y(0)=\phi_0,\,\,\,\,y_0=\phi_1,\,\,\,\Phi=(\phi_0, \phi_1)\in {\cal X},
\end{cases}
\end{equation} 
where $A$ is either assumed to generate an analytic semigroup on a Hilbert space $X=H$ as in Example 1.1 or given by a sesquilinear form $a(\cdot, \cdot)$ as in Example 1.2.
In the first case, it was shown by \cite{Gdbkkes85(2)} that when the weight function $\beta(\cdot)$ belongs to $W^{1, 2}([-r, 0]; {\mathbb R})$, the associated solution semigroup $e^{t{\cal A}}$, $t\ge 0$, is differentiable for $t>r$ or the solution semigroup is norm continuous for $t>3r$ when $\beta(\cdot)$ is H\"older continuous in the second, both of which imply further that  (\ref{20/08/2013(30)}) is fulfilled. Hence, we can describe  conditions ensuring a unique stationary solution to the equation (\ref{11/08/2013(30070)}) by showing 
\begin{equation}
\label{22/08/13(20)}
\sup\{Re\,\lambda:\, \lambda\in \sigma({\cal A})\}<0.
\end{equation}

\begin{proposition}
\label{27/08/13(40)}
 Suppose that $\sigma(A)\subset (-\infty, -c_0]$ for some $c_0>0$ and the function $\beta$ in  (\ref{11/08/2013(30070)}) satisfies  
\begin{equation}
\label{22/08/13(11)}
\|\beta\|_{L^1([-r, 0]; {\mathbb R})}<1.
\end{equation}
Then  there exists a unique stationary solution for the equation   (\ref{11/08/2013(30070)}).
\end{proposition}
\begin{proof}  Note that from Proposition \ref{22/08/13(2)} we have $\sigma({\cal A}) \subset \Gamma_0\cap \Gamma_1.$
We shall show that under the assumptions in the theorem, there is a constant  $\mu>0$ such that Re$\,\lambda\le -\mu$ for all $\lambda\in \Gamma_0\cap \Gamma_1$ and hence for all $\lambda\in \sigma({\cal A})$.

First, for elements in  $\Gamma_0$, if there exist a sequence $\{\lambda_n\}\subset {\mathbb C}$ such that Re$\,\lambda_n\ge 0$ or Re$\,\lambda_n\to 0$ as $n\to\infty$, then by (\ref{23/08/13(21)}) and Dominated Convergence Theorem, it follows that 
\[
\begin{split}
1 &= \limsup_{n\to\infty}\Big|\int^0_{-r} \beta(\theta)e^{\lambda_n\theta}d\theta\Big| \le \int^0_{-r} |\beta(\theta)|d\theta<1,
\end{split}
\]
which is clearly a contradiction. Thus the desired result is obtained.

Now we consider  elements  in $\Gamma_1$.
If there exist a sequence $\{\lambda_n\}\subset {\mathbb C}$ such that Re$\,\lambda_n\ge 0$ or Re$\,\lambda_n\to 0$ as $n\to\infty$ with $\lambda_n/n(\lambda_n) =: -\delta_n \le -c_0$, then we get by taking the real part of the equation into account that
\[
1 + \frac{Re\,\lambda_n}{\delta_n} = -\int^0_{-r} \beta(\theta)e^{(Re\,\lambda_n)\theta}\cos [(Im\,\lambda_n)\theta]d\theta.\]
Letting $n\to\infty$ and using Dominated Convergence Theorem, we get immediately that 
\[
\begin{split}
1 &\le 1 + \liminf_{n\to\infty}\frac{Re\,\lambda_n}{c_0}\le \limsup_{n\to\infty}\Big|\int^0_{-r} \beta(\theta)e^{(Re\,\lambda_n)\theta}\cos [(Im\,\lambda_n)\theta]d\theta\Big|= \int^0_{-r} |\beta(\theta)|d\theta <1,
\end{split}
\]
which, once again, yields a contradiction. Combining the above results, we thus obtain that 
\[
Re\,\lambda\le -\mu\hskip 10pt \hbox{for some}\hskip 10pt \mu>0\hskip 10pt \hbox{and all}\hskip 10pt \lambda\in\sigma({\cal A}).\]
Therefore, the solution semigroup $e^{t{\cal A}}$, $t\ge 0$, is exponentially stable. This fact further implies that there exists a unique stationary solution of  (\ref{11/08/2013(30070)}). The proof is complete.
\end{proof}

\begin{remark}\rm
The condition (\ref{22/08/13(11)}) is optimal in the sense that if we replace (\ref{22/08/13(11)}) by 
\begin{equation}
\label{27/08/13(10)}
\|\beta\|_{L^1([-r, 0]; {\mathbb R})}<1+\varepsilon\hskip 20pt \hbox{for some}\hskip 20pt \varepsilon>0,
\end{equation}
then there may not exist a unique stationary solution. Indeed, in this case let us choose $ \beta= -\frac{1+\varepsilon/2}{r}$, which clearly satisfies (\ref{27/08/13(10)}). We shall show that for such a value $\beta$, the solution system $e^{t{\cal A}}$, $t\ge 0$, could be unstable. To see this, it suffices to prove that there exists a number $\lambda=x+iy\in \Gamma_0$ with $y=0$ and $x>0$ according to Proposition  \ref{22/08/13(2)} (i). 

To this end, let us consider numbers $\lambda=x+iy\in \Gamma_0$ with $y=0$. Suppose that $\beta(\theta)\equiv \beta<0$ in (\ref{11/08/2013(30070)})
and we analyze the roots of the equation 
\begin{equation}
\label{21/08/2013(1)}
x  + \beta (1- e^{-rx})=0,\hskip20pt x\in {\mathbb R}.
\end{equation}
We put $f(x) = x  + \beta(1- e^{-rx}),$ $x\in {\mathbb R}.$
Then it is easy to see that $f'(x) = 1 + \beta re^{-rx}$, $x\in {\mathbb R}.$ 
By solving the equation $f'(x)=0$, we get $x=\ln(-\beta r)/r$ which is the unique stationary point of $f$. Since $f''(x) = -\beta r^2e^{-rx}>0$ for all $x\in {\mathbb R}$, the function $f$ takes its minimum value at  $x=\ln(-\beta r)/r$. 
 As $-\frac{1+\varepsilon/2}{r} <-1/r$,  the minimum point $x=\ln(-\beta r)/r>0$. Since $x=0$ is a  solution of (\ref{21/08/2013(1)}), the other solution $x$ of (\ref{21/08/2013(1)}) thus satisfies $x>\ln(-\beta r)/r>0$. 
\end{remark}

\begin{example}\rm
We give an application of Proposition \ref{27/08/13(40)}
to the initial-boundary value problem of Dirichlet type of the stochastic retarded Laplace equation:
\begin{equation}
\label{22/08/13(70)}
\begin{cases}
\displaystyle\frac{\partial y(t, x)}{\partial t} = \frac{\partial^2 y(t, x)}{\partial x^2} +\int^0_{-r}ae^{b\theta}\cdot   \frac{\partial^2 y(t+\theta, x)}{\partial x^2}d\theta + f(x)\dot B(t),\,\,\,\,t\ge 0,\,\,\,\,x\in {\cal O},\\
y(0, \cdot) =\phi_0(\cdot)\in W^{1, 2}_0({\cal O}; {\mathbb R}),\\
y(t, \cdot)=\phi_1(t, \cdot)\in W^{2, 2}({\cal O}; {\mathbb R})\cap W_0^{1, 2}({\cal O}; {\mathbb R}),\,\,\,\,\hbox{a.e.}\,\,\,\,t\in [-r, 0).
\end{cases} 
\end{equation}
Here ${\cal O}$ is a bounded open subset of ${\mathbb R}^n$ with regular boundary $\partial {\cal O}$, $a,\,b\in {\mathbb R}$, $r>0$ and $f\in L^2({\cal O}; {\mathbb R})$.

We can re-write (\ref{22/08/13(70)}) as a stochastic  initial boundary problem  (\ref{11/08/2013(30070)}) in the Hilbert space $H= L^2({\cal O}; {\mathbb R})$ by setting
\[
\begin{cases}
A = \displaystyle\frac{\partial^2}{\partial x^2},\\
Z={\mathscr D}(A)= W^{2, 2}({\cal O}; {\mathbb R}^1)\cap W^{1, 2}_0({\cal O}; {\mathbb R}),\\
\beta(\theta) = a e^{b\theta},\,\,\,\theta\in [-r, 0].
\end{cases}
\]
We can obtain a solution of (\ref{22/08/13(70)}) defined in $[0, \infty)$ and further apply those results derived in the section to obtain its stationary solutions. In fact, note that $A= \partial^2/\partial x^2$ is a self-adjoint and negative operator and its spectrum satisfies $\sigma(A) = \sigma_P(A) \subset (-\infty, -c_0]$ for some $c_0>0$. Then by Proposition \ref{27/08/13(40)} and a direct computation, we obtain that when
\[
|a|\le
\begin{cases}
 e^{rb}/r,\hskip 30pt &\hbox{if}\hskip 15pt b\le 0,\\
1/r,\hskip 30pt &\hbox{if}\hskip 15pt  b> 0,
\end{cases}
\]
 the associated solution semigroup of  (\ref{22/08/13(70)}) is exponentially stable. Moreover, in this case we know by Theorem \ref{16/08/2013(100)} that the equation (\ref{22/08/13(70)}) has a unique stationary solution.
\end{example}

\subsection{Discrete Delay}

Now we want to consider the following stochastic system with discrete delay on a Hilbert space $X=H$ with $Z={\mathscr D}(A)$, 
\begin{equation}
\label{11/08/2013(30567)}
\begin{cases}
\displaystyle\frac{dy(t)}{dt} =Ay(t)  + A_1 y(t-r)+ f \dot B(t),\,\,\,\,t\ge 0,\\
y(0)=\phi_0,\,\,\,\,y_0=\phi_1,\,\,\,\Phi=(\phi_0, \phi_1)\in {\cal X},
\end{cases}
\end{equation}
where $A: {\mathscr D}(A)\subset H\to H$ generates an exponentially stable, analytic semigroup $e^{tA}$, $t\ge 0$, on the Hilbert space $H$ and $A_1\in {\mathscr L}({\mathscr D}(A), H)$. In contrast with (\ref{11/08/2013(30070)}), the solution semigroup $e^{t{\cal A}}$, $t\ge 0$, of (\ref{11/08/2013(305)})  is generally not norm continuous even though $A$ generates a compact semigroup. However, if we strengthern the conditions on $A_1$, it is still possible for the associated semigroup $e^{t{\cal A}}$, $t\ge 0$, to be compact and thus one can use the spectral mapping theorem again.

\begin{lemma}
\label{19/11/13(1)}
 Assume  that  $A$ generates an exponentially stable, analytic semigroup on $H$, i.e., $\|e^{tA}\|\le Me^{-\mu t}$, $t\ge 0$, for some $M>0$, $\mu>0$. Further, if there exists a number  $\delta\in [0, 1)$ such that $A_1\in {\mathscr L}( {\mathscr D}((-A)^\delta), H)$,
then $\Delta(\lambda)^{-1}$ is compact for all $\lambda\in \rho(\Delta)$ provided that  $A$ has compact resolvents.
\end{lemma}
\begin{proof} For arbitrary $\lambda\in {\mathbb C}$, we define $F_\lambda:\,  {\mathscr D}((-A)^\delta)\to H$ by $F_\lambda x := F(e^{\lambda\cdot }x)$ for $x\in {\mathscr D}((-A)^\delta)$. It is easy to see that $F_\lambda\in {\mathscr L}({\mathscr D}((-A)^\delta), H)$ (thus, $F_\lambda\in {\mathscr L}({\mathscr D}(A), H))$. By Corollary 6.11, p. 73, Pazy \cite{ap83} there is a constant $C>0$ such that for every $\rho>0$,
\begin{equation}
\label{04/11/13(1)}
\|F_\lambda x\|_H\le C(\rho^\delta\|x\|_H + \rho^{\delta-1}\|Ax\|_H),\hskip20pt \forall\,x\in {\mathscr D}(A).
\end{equation}
This implies that $F_\lambda$ is $A$-bounded with $A$-bound $0$ (see Pazy \cite{ap83}).
Hence,  $A+F_\lambda$ generates a $C_0$-semigroup on $H$. In particular, $\mu\in \rho(A + F_\lambda)$ for Re$\, \mu$ large enough.

Moreover, $\mu\in \rho(A)$ if Re$\,\mu$ is large enough and we have for any fixed $\lambda\in {\mathbb C}$ that
\[
(\mu I -A-F_\lambda) = (I- F_\lambda R(\mu, A))(\mu I -A).\]
Hence,  if we can show that $\|F_\lambda  R(\mu, A)\|_{{\mathscr L}(H)}<1$,
Re$\,\mu\ge R$, for some $R>0$, then it is true that  
\begin{equation}
\label{19/07/2013(1)}
(\mu I -A-F_\lambda)^{-1}= R(\mu, A)[I -F_\lambda  R(\mu, A)]^{-1}\hskip 10pt \hbox{for any}\,\,\, {Re}\,\mu\ge \delta_0. 
\end{equation}
To this end, we recall that the analyticity of $e^{tA}$ implies that there exists a constant $M>0$ such that 
\[
\|A R(\mu, A)\|_{{\mathscr L}(H)}\le M\hskip 15pt \hbox{for large }\hskip 15pt Re\,\mu.\]
Let $0<\varepsilon <1$ and $a<\varepsilon/2M$. It follows by (\ref{04/11/13(1)}) that there exists $b>0$ such that 
\[
\|F_\lambda R(\mu, A)x\|_H \le a \|A R(\mu, A)x\|_H + b\|R(\mu, A)x\|_H,\hskip 15pt \forall x\in H.\]
Now choose Re$\, \mu$  large enough so that 
\[
b\|R(\mu, A)\|_{{\mathscr L}(H)}<{\varepsilon}/{2}.\]
Then it is easy to obtain that 
\[
\|F_\lambda R(\mu, A)\|_{{\mathscr L}(H)}\le \varepsilon <1 \hskip 15pt \hbox{for large}\,\,\,\,Re\,\mu\in {\mathbb R}.\]
Since $R(\mu, A)$ is compact on $H$, so is $(\mu I -A-F_\lambda)^{-1}$ according to (\ref{19/07/2013(1)}). Last, let $\mu=\lambda\in {\mathbb C}$, then $(\lambda I -A-F_\lambda)^{-1}$ with large Re$\,\lambda$ (then, for all $\lambda\in \rho(\Delta)$) are compact, and the desired result is concluded.
\end{proof}

 Let $\lambda\in {\mathbb C}$ and we introduce mappings ${E}_\lambda:\, Z\to {\cal X}$, ${J}_\lambda:\, {\cal X}\to {\cal X}$  and ${H}_\lambda:\, W\times L^2([-r, 0]; H)\to H$, respectively, by 
 \begin{equation}
\label{30/06/06(10)}
\begin{cases}
({E}_\lambda x)_0 =x,\\
({E}_\lambda x)_1(\theta) =e^{\lambda\theta}x,\hskip 10pt \theta\in [-r, 0],\hskip 10pt \hbox{for}\hskip 10pt x\in Z,
\end{cases}
\end{equation} 
 \begin{equation}
\label{30/06/06(11)}
\begin{cases}
({J}_\lambda \Phi)_0 =0,\\
({J}_\lambda \Phi)_1(\theta) =\displaystyle\int^0_\theta e^{\lambda(\theta-\tau)}\phi_1(\tau)d\tau,\hskip 10pt \theta\in [-r, 0],\hskip 10pt \hbox{for}\hskip 10pt \Phi\in {\cal X},
\end{cases}
\end{equation} 
 \begin{equation}
\label{30/06/06(13)}
({H}_\lambda \Phi) =\phi_0 + \int^0_{-r} e^{\lambda\tau}\phi_1(\tau)d\tau,\hskip 10pt \hbox{for}\hskip 10pt \Phi\in W\times L^2([-r, 0]; H).
\end{equation} 
 It is immediate to know  that all the four operators ${E}_\lambda$, ${J}_\lambda$, ${K}_\lambda$ and ${H}_\lambda$ are linear and bounded. 

\begin{lemma}
\label{20/11/13(60)}
Suppose that $A$ generates a compact semigroup $e^{tA}$ for $t>0$. Under the same conditions as in Lemma \ref{19/11/13(1)}, it is true that $R(\lambda, {\cal A})e^{r{\cal A}}$ is compact for some $\lambda\in \rho({\cal A})$.
\end{lemma}
\begin{proof} By definition, it is not difficult to get that 
\begin{equation}
\label{19/11/13(10)}
R(\lambda, {\cal A}) = E_\lambda \Delta(\lambda)^{-1}H_\lambda {\cal S} + J_\lambda,\hskip 15pt \lambda\in \rho({\cal A}),
\end{equation}
which immediately implies that $\rho(\Delta)=\rho({\cal A})$. Here ${\cal S}$ is the structure operator given in (\ref{21/11/13(95)}).
Let $\pi_0:\, W\times L^2([-r, 0]; {\mathscr D}(A))\to W$ and $\pi_1:\, W\times L^2([-r, 0]; {\mathscr D}(A))\to L^2([-r, 0]; {\mathscr D}(A))$ denote the canonical projections on $W$ and $L^2([-r, 0]; {\mathscr D}(A))$, respectively. 

Since $\Delta(\lambda)^{-1}$ is compact for all $\lambda\in \rho(\Delta)$ by virtue of Lemma \ref{19/11/13(1)}, we have by using (\ref{19/11/13(10)}) and the compactness of $\Delta(\lambda)^{-1}$ to get  that 
\begin{equation}
\label{20/11/13(1)}
\pi_0 R(\lambda, {\cal A})e^{r{\cal A}}=\pi_0 E_\lambda \Delta(\lambda)^{-1} H_\lambda {\cal S} e^{r{\cal A}}
\end{equation}
 is compact. 

Now we restrict our attention to $\pi_1 R(\lambda, {\cal A})e^{r{\cal A}}:\, W\times L^2([-r, 0]; {\mathscr D}(A))\to L^2([-r, 0]; {\mathscr D}(A))$ for any fixed $\lambda\in \rho({\cal A})$.  Note that 
\[
\frac{d}{d\theta}\pi_1 R(\lambda, {\cal A}) e^{r{\cal A}} = \pi_1 {\cal A} R(\lambda, {\cal A})e^{r{\cal A}},\]
and $\|e^{r{\cal A}}\|\le M$, $\|{\cal A}R(\lambda, {\cal A})\|\le M$ for some $M>0$.
Hence,  for any $\Phi=(\phi_0, \phi_1)\in {\cal X}$ we can deduce by using H\"older's inequality  that for all $\theta_1,\,\theta_2\in [-r, 0]$,
\[
\begin{split}
\|[\pi_1 R(\lambda, {\cal A}) e^{r{\cal A}}\Phi](\theta_2)& -[\pi_1 R(\lambda, {\cal A}) e^{r{\cal A}}\Phi](\theta_1)\|_{{\mathscr D}(A)}\\
&= \Big\|\int^{\theta_2}_{\theta_1}\Big[\frac{d}{d\theta}\pi_1 R(\lambda, {\cal A})e^{r{\cal A}}\Phi\Big](\theta)d\theta\Big\|_{{\mathscr D}(A)}\\
&= \Big\|\int^{\theta_2}_{\theta_1}[\pi_1 {\cal A}R(\lambda, {\cal A})e^{r{\cal A}}\Phi](\theta)d\theta\Big\|_{{\mathscr D}(A)}\\
&\le \int^{\theta_2}_{\theta_1}\|[\pi_1{\cal A} R(\lambda, {\cal A})e^{r{\cal A}}\Phi](\theta)\|_{{\mathscr D}(A)} d\theta\\
&\le (\theta_2-\theta_1)^{1/2}\Big\|\pi_1{\cal A}R(\lambda, {\cal A})e^{r{\cal A}}\Phi\Big\|_{L^2([-r, 0]; {{\mathscr D}(A)})}\\
&\le M|\theta_2-\theta_1|^{1/2}\|\Phi\|_{\cal X}.   
\end{split}
\] 
This implies that the family 
\begin{equation}
\label{20/11/13(15)}
\Sigma := \Big\{\pi_1 R(\lambda, {\cal A})e^{r{\cal A}}\Phi:\,\,\Phi\in {\cal X}, \|\Phi\|_{\cal X}\le 1\Big\}\subset C([-r, 0]; {{\mathscr D}(A)})
\end{equation}
is equi-continuous. On the other hand, we have for any $\theta\in [-r, 0]$ that 
\begin{equation}
\label{20/11/13(31)}
\begin{split}
[\pi_1R(\lambda, {\cal A})e^{r{\cal A}}\Phi](\theta) &= [\pi_1e^{r{\cal A}}R(\lambda, {\cal A})\Phi](\theta)\\
&= [\pi_1 e^{(r+\theta){\cal A}}R(\lambda, {\cal A})\Phi](0)\\
&= [\pi_1R(\lambda, {\cal A})e^{(r+\theta){\cal A}}\Phi](0)\\
&= \pi_0 R(\lambda, {\cal A})e^{(r+\theta){\cal A}}\Phi.
\end{split}
\end{equation}
By virtue of (\ref{19/11/13(10)}), (\ref{20/11/13(31)}) and the fact that $\Delta(\lambda)^{-1}$ is compact, we get that $\Sigma$ in (\ref{20/11/13(15)}) is pointwise relatively compact. Hence, we find by virtue of Ascoli-Arzel\`a theorem that $\Sigma$ is relatively compact in $C([-r, 0]; {{\mathscr D}(A)})$ and further relatively compact in $L^2([-r, 0]; {{\mathscr D}(A)})$. From this we conclude that $\pi_1 R(\lambda, {\cal A})e^{r{\cal A}}$ is compact which, in addition to (\ref{20/11/13(1)}), implies the compactness of $R(\lambda, {\cal A})e^{r{\cal A}}$. The proof is thus complete. 
\end{proof}

\begin{theorem}
Under the same conditions as in Lemma \ref{19/11/13(1)}, we have that the semigroup $e^{t{\cal A}}$, $t\ge 0$, is compact for all $t>r$ provided that $A$ generates a compact semigroup $e^{tA}$ for $t>0$.
\end{theorem}
\begin{proof}
It suffices to show that $e^{t{\cal A}}$, $t\ge 0$, is norm continuous for $t>r$ and $R(\lambda, {\cal A})e^{r{\cal A}}$ is compact for some $\lambda\in \rho({\cal A})$ (see \cite{kern00}, Lemma II, 4.28).

Since $e^{tA}$ is  compact (thus, norm continuous) for $t>0$, by a similar argument to Proposition 6.2 in \cite{kl09(2)}, it is possible to show that $e^{t{\cal A}}$ is norm continuous for $t>r$. In addition to Lemma \ref{20/11/13(60)}, it follows that $e^{t{\cal A}}$ is compact for all $t>r$. The proof is complete now.
\end{proof}

\begin{example}\rm
Consider the following stochastic partial differential equation with delay
\begin{equation}
\label{21/11/13(90)}
\begin{cases}
\partial y(x, t)/\partial t = \Delta y(x, t) + \alpha (-\Delta)^{\delta}y(x, t-1) + f(x)dB(t), \,\,\,x\in {\cal O},\,\,\,\, t\ge 0,\\
y(x, t) =0,\,\,\,\,x\in \partial{\cal O},\,\,\,t\ge 0,\\
y(x, t) =\varphi(x, t),\,\,\,\,(x, t)\in {\cal O}\times [-1, 0],
\end{cases}
\end{equation}
where $\Delta$ is the standard Laplacian operator, $\delta\in [0, 1)$, $\alpha\in {\mathbb R}$ and ${\cal O}\subset {\mathbb R}^N$ a bounded open set with smooth boundary. Let $H= L^2({\cal O})$ and  the Dirichlet-Laplacian 
\[
A = \Delta\,\,\hbox{with domain}\,\,{\mathscr D}(A) = \{u\in H^1_0({\cal O}):\, \Delta u\in L^2({\cal O})\}.
\]
It is claimed that  the equation (\ref{21/11/13(90)}) has a unique stationary solution  if 
\[
2|\alpha|<|\lambda_1|^{1-\delta}\]
where $\lambda_1$ is the first eigenvalue of the Dirichlet-Laplacian. 

Indeed, we have 
\[
\begin{cases}
F = \alpha (-\Delta)^\delta\bm{\delta}_{-1},\\
F_\lambda =\alpha e^{-\lambda}(-\Delta)^{\delta}.
\end{cases}
\]
Here we define $\bm{\delta}_{-1}: C([-1, 0]; {\mathscr D}((-\Delta)^\delta))\to H$ by $\bm{\delta}_{-1}(\varphi) =\varphi(-1)$. Since $A$ is a self-adjoint operator on $H$, we can compute for $a\in {\mathbb R}$ that 
\[
\begin{split}
\|F_{ia}R(ia, A)\| &\le \alpha \|(-A)^{\delta} R(ia, A)\|\\
&\le \alpha \|(-A)^{\delta-1}\|\|A R(ia, A)\|\\
&\le \alpha \|(-A)^{\delta-1}\| ( |a|\|R(ia, A)\|+1)\\
&= \alpha \|(-A)^{\delta-1}\|\Big(\frac{|a|}{\sqrt{a^2 + \lambda^2_1}} +1\Big),
\end{split}
\]
which yields immediately that
\[
\sup_{a\in {\mathbb R}}\|F_{ia} R(ia, A)\|\le \frac{2|\alpha|}{|\lambda_1|^{1-\delta}}<1.\]
By virtue of Phragmen-Lindel\"of Theorem (see Conway \cite{JC86}, Theorem VI. 4.1), it follows that 
\begin{equation}
\label{21/11/13(1)}
\sup_{Re\,\lambda\ge 0}\|F_\lambda R(\lambda, A)\|<1.
\end{equation}
The relation (\ref{21/11/13(1)}) ensures the existence of $\Delta(\lambda)^{-1}$ on the halfplane $\{Re\,\lambda\ge 0\}$ which is given by the Neumann series 
\[
\Delta(\lambda)^{-1} = R(\lambda, A)\sum^\infty_{n=0} (F_\lambda R(\lambda, A))^n.\]
Hence, we have $\sup\{Re\,\lambda:\, \lambda \in \sigma({\cal A})\}<0.$
Since the associated semigroup $e^{t{\cal A}}$ is norm continous for $t>r$,  the growth bound of ${\cal A}$ thus satisfies
\[
\inf\{\mu: \,\|e^{t{\cal A}}\|\le Me^{\mu t}\,\hbox{for some}\,\,M>0\}=
  \sup\{Re\,\lambda:\, \lambda \in \sigma({\cal A})\}<0.
\]
That is, the solution semigroup $e^{t{\cal A}}$, $t\ge 0$, is exponentially stable, a fact which assures the existence of a unique a stationary solution of the equation (\ref{21/11/13(90)}).
\end{example}

Now we return to consider the equation (\ref{11/08/2013(30567)}) with $A_1=\alpha A$, $\alpha\in {\mathbb R}$, i.e., 
\begin{equation}
\label{11/08/2013(305)}
\begin{cases}
\displaystyle\frac{dy(t)}{dt} =Ay(t)  + \alpha A y(t-r)+ f \dot B(t),\,\,\,\,t\ge 0,\\
y(0)=\phi_0,\,\,\,\,y_0=\phi_1,\,\,\,\Phi=(\phi_0, \phi_1)\in {\cal X},
\end{cases}
\end{equation}
where $A: {\mathscr D}(A)\subset H\to H$ generates an analytic semigroup $e^{tA}$, $t\ge 0$, on the Hilbert space $H$. On this occasion,  the solution semigroup $e^{t{\cal A}}$, $t\ge 0$, of (\ref{11/08/2013(305)})  is never compact, or even  norm continuous. A direct consequence of this fact is that one cannot use the standard spectral mapping theorem to obtain stationary solutions for  Equation  (\ref{11/08/2013(305)}). 

In the sequel, we will employ a different method by estimating the growth bound through some resolvent estimates. More precisely, we estimate the growth bound by considering the abscissa of uniform boundedness of the resolvent of the generator ${\cal A}$ (cf. \cite{absp05}). 

Suppose that  $B$ is the infinitesimal generator of an arbitrary $C_0$-semigroup on the Hilbert space $H$ and $s(B)$ is defined as the infimum of all $\mu\in {\mathbb R}$ such that $\{\hbox{Re}\,\lambda>\mu\}\subset \rho(B)$ and $\sup_{Re\,\lambda>\mu}\|R(\lambda, B)\|<\infty$, then (see, e.g., \cite{kern00}) 
\[
s(B) = \inf \{\mu\in {\mathbb R}:\, \|e^{tB}\|\le Me^{\mu t}\,\,\hbox{for some}\,\,M>0\}.\] 
Moreover, if the generator $B$ satisfies the conditions of Gearhart-Pr\"uss-Greiner Theorem:
\begin{equation}
\label{02/09/13(11)}
\{\lambda\in {\mathbb C}:\, \hbox{Re}\,\lambda>0\}\subset \rho(B)\hskip 10pt \hbox{and}\hskip 10pt \sup_{Re\,\lambda>0}\|R(\lambda, B)\|<\infty,
\end{equation}
then $s(B)<0$ and the semigroup $e^{tB}$, $t\ge 0$, is thus exponentially stable (cf.  \cite{kern00}).

We will consider the spectrum $\sigma({\cal A})$ and the resolvent $R(\lambda, {\cal A})$ of the solution semigroup $e^{t{\cal A}}$, $t\ge 0$, of the equation (\ref{11/08/2013(305)}). Recall that the characteristic operator $\Delta(\lambda): {\mathscr D}(A)\to H$ for   (\ref{11/08/2013(305)}) is given on this occasion by 
\[
\Delta(\lambda)x = \lambda x - n(\lambda) Ax,\hskip 20pt x\in {\mathscr D}(A),\]
where $n(\lambda)= 1 + \alpha e^{-\lambda r}$, $\lambda\in {\mathbb C}.$ 

\begin{proposition}
\label{30/08/13(5)}
For the equation (\ref{11/08/2013(305)}), assume that $\sigma(A)\subset (-\infty, -c_0]$ for some $c_0>0$ and $|\alpha|<1$, then it is valid that
\[
\sigma({\cal A})\subset \{\lambda\in {\mathbb C}:\, Re\,\lambda<0\}.\]
\end{proposition}
 \begin{proof}
Since $|\alpha|<1$, it follows that $n(0)\not= 0$ and by Proposition \ref{22/08/13(2)}, $\sigma({\cal A})\subset \Gamma_0\cap \Gamma_1$.  We first assume $\lambda\in \Gamma_1$, then there is a $\gamma\in \sigma(A)$ such that $\lambda/n(\lambda)=\gamma<0$. Let us denote $\lambda=x+iy\in {\mathbb C}$. Then the real part of the equation yields that 
\begin{equation}
\label{30/08/13(1)}
\frac{x}{1+\alpha e^{-rx}\cos ry}=\gamma<0.
\end{equation}
If $x\ge 0$, then it follows by assumption that 
\[
|\alpha e^{-rx}\cos ry|\le |\alpha|<1.\]
This implies that $1+\alpha e^{-r x}\cos ry>0$, a fact which
contradicts with (\ref{30/08/13(1)}), thus $x<0$.

Now let $\lambda\in \Gamma_0$, then (\ref{23/08/13(21)}) and (\ref{01/09/13(10)})
 imply that 
\begin{equation}
 \label{30/08/13(2)}
1+\alpha e^{-rx}\cos ry=0.
\end{equation}
If $x\ge 0$, we have from  (\ref{30/08/13(2)}) that 
\[
1 = |\alpha e^{-rx}\cos ry|\le |\alpha|<1,\]
which is a contradiction again. Combining the above arguments and using Proposition 
\ref{22/08/13(2)}, we obtain the desired results. The proof is complete now.
\end{proof}

Now we are in a position to obtain the stationary solutions of the equation (\ref{11/08/2013(305)}). To this end, we first present a useful lemma.
\begin{lemma}
\label{30/08/13(11)} If there exists a constant $C>0$ such that  for any $\lambda\in \{\lambda\in {\mathbb C}:\, Re\,\lambda>0\}$, 
\begin{equation}
\label{29/08/13(1)}
\|x\|_{{\mathscr D}(A)}\le C\|\Delta(\lambda)x\|_{H}\hskip 15pt \hbox{for each}\hskip 15pt x\in {\mathscr D}(A),
\end{equation}
then there exists a constant $M>0$ such that 
\[
\|\Phi\|_{\cal X} \le M\|(\lambda I - {\cal A})\Phi\|_{\cal X}\hskip 15pt \hbox{for each }\hskip 15pt \Phi\in {\mathscr D}({\cal A}).\]
\end{lemma}
\begin{proof}
First observe that for arbitrary $x\in {\mathscr D}(A)$ and $y\in L^2([-r, 0]; {\mathscr D}(A))$, the function 
\[
u(\theta) = e^{\lambda\theta}x + \int^0_{\theta}e^{\lambda(\theta-\tau)}y(\tau)d\tau,\hskip 20pt \theta\in [-r, 0],\]
satisfies that
\begin{equation}
\label{29/08/13(11)}
\begin{split}
\|u\|^2_{L^2([-r, 0]; {\mathscr D}(A))} \le &\,\, 2\|e^{\lambda\cdot}x\|_{L^2([-r, 0]; {\mathscr D}(A))}^2 + 2\Big\|\int^0_\cdot e^{\lambda(\cdot-\tau)}y(\tau)d\tau\Big\|^2_{L^2([-r, 0]; {\mathscr D}(A))}\\
\le &\,\, 2r\|x\|^2_{{\mathscr D}(A)} + 2r^2\|y\|^2_{L^2([-r, 0]; {\mathscr D}(A))}.
\end{split}
\end{equation}
For any $\Phi=(\phi_1(0), \phi_1)\in {\mathscr D}({\cal A})$, we set $\Psi=\lambda\Phi -{\cal A}\Phi$ and let $x=\phi_1(0)\in {\mathscr D}(A)$. By virtue of (\ref{28/07/2013(1)}), the inequality (\ref{29/08/13(1)}) implies that
\begin{equation}
\label{29/08/13(12)}
\begin{split}
\|\phi_1(0)\|^2_{{\mathscr D}(A)} \le &\,\, C\|\Delta(\lambda)\phi_1(0)\|^2_{H}\\
=&\,\, C\Big\|\alpha \int^0_{-r}e^{\lambda(-r-\tau)}A\psi_1(\tau)d\tau + \psi_0\Big\|^2_{H}\\
\le &\,\, 2C\{|\alpha|^2 r\|A\|_{{\mathscr L}({\mathscr D}(A), H)}^2\|\psi_1\|^2_{L^2([-r, 0]; {\mathscr D}(A))} + \|\psi_0\|^2_{H}\}.
\end{split}
\end{equation}
Since ${\mathscr D}(A)\hookrightarrow W\hookrightarrow H$, there exists a constant $\gamma>0$ such that 
\begin{equation}
\label{25/10/11(2)}
\|u\|_W \le \gamma \|u\|_{{\mathscr D}(A)},\hskip 15pt \|v\|_{H}\le \gamma \|v\|_W\hskip 15pt \hbox{for any}\hskip 15pt u\in {\mathscr D}(A), \hskip 5pt v\in W.
\end{equation}
Hence,  from (\ref{25/10/11(2)}), (\ref{29/08/13(11)}) and (\ref{29/08/13(12)}) it follows that for arbitrary $\Phi\in {\mathscr D}({\cal A})$, 
\begin{equation}
\label{02/09/13(20)}
\begin{split}
\|\Phi\|^2_{\cal X} & = \|\phi_1(0)\|^2_W + \|\phi_1\|^2_{L^2([-r, 0]; {\mathscr D}(A))}\\
&\le \gamma^2\|\phi_1(0)\|^2_{{\mathscr D}(A)} + 2r\|\phi_1(0)\|^2_{{\mathscr D}(A)} + 2r^2\|\psi_1\|^2_{L^2([-r, 0]; {\mathscr D}(A))}\\
&\le 2C(\gamma^2 +2r)\Big\{|\alpha|^2r \|A\|^2_{{\mathscr L}({\mathscr D}(A), H)}\|\psi_1\|^2_{L^2([-r, 0]; {\mathscr D}(A))} + \|\psi_0\|^2_{H}\Big\} + 2r^2\|\psi_1\|^2_{L^2([-r, 0]; {\mathscr D}(A))},
\end{split}
\end{equation}
which, together with (\ref{25/10/11(2)}) and (\ref{02/09/13(20)}), further implies the existence of a constant $M>0$ such that 
\[
\|\Phi\|^2_{\cal X} \le M\{\|\psi_1\|^2_{L^2([-r, 0]; {\mathscr D}(A))} + \|\psi_0\|^2_W\} = M\|\Psi\|^2_{\cal X} =M\|\lambda\Phi - {\cal A}\Phi\|^2_{\cal X}\]
for any $\Phi\in {\mathscr D}({\cal A})$.
The proof is complete now.
\end{proof}

\begin{proposition}
\label{02/09/13(60)} Suppose that $A$ is a self-adjoint operator on $H$. 
Under the same conditions as in Proposition \ref{30/08/13(5)}, there exists a unique stationary solution of the equation (\ref{11/08/2013(305)}).
\end{proposition}
\begin{proof}
We show that under the conditions  in Proposition \ref{02/09/13(60)}, the associated solution semigroup $e^{t{\cal A}}$, $t\ge 0$, of (\ref{11/08/2013(305)}) is exponentially stable.

Indeed, it is clear that the inverse of the characteristic operator $\Delta(\lambda)= \lambda I - n(\lambda)A$ exists and $\Delta(\lambda)^{-1}\in {\mathscr L}(H)$ whenever 
\[
n(\lambda)\not= 0\hskip 20pt\hbox{and}\hskip 20pt \Big(\frac{\lambda}{n(\lambda)} I - A\Big)^{-1}\in {\mathscr L}(H).\]
In this case, the inverse is actually given by 
\begin{equation}
\label{02/09/13(50)}
\Delta(\lambda)^{-1} =\frac{1}{n(\lambda)}\Big(\frac{\lambda}{n(\lambda)} I - A\Big)^{-1}.
\end{equation}
Let $z =\lambda/n(\lambda)$ and $R(z, A) = (\frac{\lambda}{n(\lambda)} I - A)^{-1}$. We shall study the operator $R(z, A)$ with Re$\,\lambda>0$. We  show that for Re$\,\lambda>0$, $z\in \Sigma\subset \rho(A)$ where 
\[
\Sigma := \Big\{\lambda\in {\mathbb C}:\, |\hbox{arg}\,\lambda|<\frac{\pi}{2} + \theta\Big\}\hskip 20pt \hbox{for some}\hskip 20pt \theta\in (0, \pi/2).\]
 Let us denote $\lambda=x+iy\in {\mathbb C}$ and assume $x>0$. By definition we have
\[
\hbox{Im}\,n(\lambda) = \alpha e^{-rx}\sin ry\hskip 20pt \hbox{and}\hskip 20pt \hbox{Re}\,n(\lambda)= 1 +  \alpha e^{-rx}\cos ry.\]
Since $|e^{-rx}\sin ry|\le 1$ and 
$\alpha e^{-rx}\cos ry\le |\alpha|<1$, we obtain that  
\[
1-|\alpha|<|1 + \alpha  e^{-rx}\cos ry|,\]
and further 
\[
\frac{|\hbox{Im}\,n(\lambda)|}{|\hbox{Re}\,n(\lambda)|}\le \frac{|\alpha|}{ 1 - |\alpha|}<\infty.\]
This means that 
\[
|\hbox{arg}\,n(\lambda)|<\theta<\frac{\pi}{2}\hskip 20pt \hbox{and}\hskip 20pt |\hbox{arg}\,z|< \frac{\pi}{2}+\theta<\pi.\] 
By assumption, $A$ is a self-adjoint operator  so that we can  obtain from the spectral theory of operators (see Kato   \cite{kato80}, Section V. 3.8) that
\begin{equation}
\label{02/09/13(52)}
\|R(z, A)\|_{{\mathscr L}(H)}=\sup_{a\in \sigma(A)}\frac{1}{|a-z|}\le \frac{1}{d},
\end{equation}
where $d=\hbox{dist}(\sigma(A), \Sigma)>0$. Thus both (\ref{02/09/13(50)}) and  (\ref{02/09/13(52)}) imply that  
\[
\|\Delta(\lambda)^{-1}\|_{{\mathscr L}(H)}\le \frac{\|R(z, A)\|_{{\mathscr L}(H)}}{|n(\lambda)|}\le \frac{1}{d(1-|\alpha|)}<\infty.\]
Now we can use Lemma \ref{30/08/13(11)} and Gearhart-Pr\"uss-Greiner Theorem to conclude the exponential stability of the solution semigroup $e^{t{\cal A}}$, $t\ge 0$ and further obtain by Theorem \ref{16/08/2013(100)} a unique stationary solution of  (\ref{11/08/2013(305)}). The proof is thus complete.
\end{proof}

\begin{example}
\label{18/11/13(1)}
\rm
We consider a stochastic partial integro-differential equation with delays in the highest-order derivatives,
\begin{equation}
\label{22/08/13(7078)}
\begin{cases}
\displaystyle\frac{\partial y(t, x)}{\partial t} = \frac{\partial^2 y(t, x)}{\partial x^2} +\alpha    \frac{\partial^2 y(t-r, x)}{\partial x^2} + f(x)\dot B(t),\,\,\,\,t\ge 0,\,\,\,\,x\in {\cal O},\\
y(0, \cdot) =\phi_0(\cdot)\in W^{1, 2}_0({\cal O}; {\mathbb R}),\\
y(t, \cdot)=\phi_1(t, \cdot)\in W^{1, 2}_0({\cal O}; {\mathbb R})\cap W^{2, 2}({\cal O}; {\mathbb R}),\,\,\,\,\hbox{a.e.}\,\,\,\,t\in [-r, 0).
\end{cases} 
\end{equation}
Here ${\cal O}$ is a bounded open subset of ${\mathbb R}^n$ with regular boundary $\partial {\cal O}$, $\alpha\in {\mathbb R}$, $r>0$ and $f\in L^{2}({\cal O}; {\mathbb R})$.

By analogy with Example 5.2, we can re-write (\ref{22/08/13(7078)}) as a stochastic  initial boundary problem  (\ref{11/08/2013(305)}) in the Hilbert space $H= L^2({\cal O}; {\mathbb R})$ to obtain a solution defined in $[0, \infty)$. In particular, since $A= \partial^2/\partial x^2$ is a self-adjoint and negative operator and its spectrum satisfies $\sigma(A) = \sigma_P(A) \subset (-\infty, -c_0]$ for some $c_0>0$. Then by Proposition \ref{02/09/13(60)} and a direct computation, we may obtain that when $|\alpha |<1,$
 the associated solution semigroup of  (\ref{22/08/13(7078)}) is exponentially stable, and further by Theorem \ref{16/08/2013(100)} the equation (\ref{22/08/13(7078)}) has a unique stationary solution.
\end{example}

\section{Appendix}
{\bf Proof of Proposition 1.1}. 
For fixed $T\ge 0$ and any $y\in C([-r, T]; Z)$, one can get by using (\ref{14/10/13(1)}), H\"older inequality and Fubini's theorem that
\begin{equation}
\label{14/10/13(2)}
\begin{split}
\int^T_0 \|Fy_t\|^2_X dt &= \int^T_0 \Big\|\int^0_{-r} d\eta(\theta)y(t+\theta)\Big\|^2_X dt\\
&\le \int^T_{0}\Big(\int^0_{-r} \|y(t+\theta)\|_Z d|\eta|(\theta)\Big)^2dt\\
&\le |\eta|([-r, 0]) \int^T_0 \int^0_{-r} \|y(t+\theta)\|_Z^2 d|\eta|(\theta) dt\\
&\le  |\eta|([-r, 0]) \int^0_{-r} \int^T_{-r} \|y(t)\|^2_Z dt d|\eta|(\theta)= |\eta|([-r, 0])^2 \int^T_{-r} \|y(t)\|^2_Z dt,
\end{split}
\end{equation}
where  $|\eta|([-r, 0])$ is the total variation of $\eta$ on $[-r, 0]$.
Since $C([-r, T]; Z)$ is dense in $L^2([-r, T]; Z)$,  the delay operator ${F}$  is extendible so that (\ref{14/10/13(2)}) remains true for all $y\in L^2([-r, T]; Z)$
and the positive constant $C$ in (\ref{14/10/13(3)}) is  given by $C= |\eta|([-r, 0])^2>0$.

\noindent {\bf Proof of Proposition \ref{14/08/2013(70)}}.
The equation (\ref{28/07/2013(2)}) can be equivalently written as
\begin{equation}
\label{28/07/2013(10)}
\lambda \phi_1(0) -A\phi_1(0) -F\phi_1=\psi_0,
\end{equation}
\begin{equation}
\label{28/07/2013(4)}
\lambda \phi_1(\theta) -{d\phi_1(\theta)}/{d\theta}= \psi_1(\theta)\hskip 15pt \hbox{for}\,\,\,\,\theta\in [-r, 0],
\end{equation}
and further (\ref{28/07/2013(4)}) is equivalent to (\ref{28/07/2013(5)}). Hence if (\ref{28/07/2013(2)}) holds we deduce that $\phi_1(0)\in {\mathscr D}(A)$ and that (\ref{28/07/2013(1)}) is true by virtue of (\ref{28/07/2013(10)}) and (\ref{28/07/2013(5)}). 

Conversely, if $\phi_0\in {\mathscr D}(A)$ then $\phi_1$ defined by (\ref{28/07/2013(20)}) belongs to $W^{1, 2}([-r, 0]; Z)$. If, in addition, $\phi_1(0)=\phi_0$ satisfies (\ref{28/07/2013(1)}) then from (\ref{28/07/2013(20)}) we get 
\begin{equation}
\label{28/07/2013(50)}
\begin{split}
\lambda \phi_1(0) -A\phi_1(0)&= F(e^{\lambda\cdot})\phi_1(0) + \psi_0 + F\Big(\int^0_\cdot e^{\lambda(\cdot-\tau)}\psi_1(\tau)d\tau\Big)\\
&= F(e^{\lambda\cdot})\phi_1(0) + \psi_0 + F\Big(\int^0_\cdot e^{\lambda(\cdot-\tau)}\Big[\lambda\phi(\tau) - \frac{d\phi_1(\tau)}{d\tau}\Big]d\tau\Big)\\
&= F(e^{\lambda\cdot})\phi_1(0) + \psi_0 -F[e^{\lambda\cdot}\phi_1(0)] + F\phi_1\\
&= \psi_0 +F\phi_1,
\end{split}
\end{equation}
that is, 
\[
\lambda \phi_1(0) -A\phi_1(0) -F\phi_1 =\psi_0,\]
which is exactly the relation (\ref{28/07/2013(2)}) and the proof is thus complete.

\end{document}